\documentclass[12pt]{amsart}

\usepackage{calc}
\newlength{\barwidth}
\usepackage{tikz}
\usetikzlibrary{shapes.multipart,calc,shapes.geometric,positioning}

\tikzset{
    horizontal stripes/.style={
        rectangle split, 
        rectangle split parts=3,
        draw,
        minimum width=15pt*\flrat,
        minimum height=15pt,
        rectangle split part fill={#1},
        inner sep=0pt,outer sep=0pt,
        name=flag},
    vertical stripes/.style={
        rectangle split,
        rectangle split horizontal, 
        rectangle split parts=3,
        draw,
        minimum width=15*\flrat,
        minimum height=15pt,
        text width=15pt*\flrat/3,
        rectangle split part fill={#1},
        inner sep=0pt,outer sep=0pt,
        name=flag},
    two stripes/.style={
        rectangle split, 
        rectangle split parts=2,
        draw,
        minimum width=150pt,
        minimum height=100pt,
        rectangle split part fill={#1},
        inner sep=0pt,outer sep=0pt}
    }

\definecolor[named]{FrenchBlue}{RGB}{0,85,164}
\definecolor[named]{FrenchRed}{RGB}{239,65,53}
\definecolor[named]{RussianBlue}{RGB}{0,57,166}
\definecolor[named]{RussianRed}{RGB}{213,43,30}

 \numberwithin{equation}{section}

\usepackage{xcolor}
\usepackage{tikz}
\usepackage{graphicx}
\usepackage{youngtab}
\usepackage[vcentermath]{genyoungtabtikz}
\usepackage{amsmath}
\usepackage{amsfonts}
\usepackage{stmaryrd}
\usepackage{extarrows}

\newcommand{\eq}{\begin{equation}}
\newcommand{\eeq}{\end{equation}}

\newtheorem{thm}{Theorem}

\newtheorem*{ques}{Big Question}
\newtheorem{corollary}{Corollary}
\newtheorem{lem}{Lemma}
\newtheorem{obs}{Observation}
\newtheorem{prop}{Proposition}
\theoremstyle{definition}

\newtheorem{remark}{Remark}

\DeclareMathOperator{\av}{\mathbf{a}}
\DeclareMathOperator{\bv}{\mathbf{b}}
\DeclareMathOperator{\cv}{\mathbf{c}}
\DeclareMathOperator{\dv}{\mathbf{d}}

\DeclareMathOperator{\uv}{\mathbf{u}}
\DeclareMathOperator{\vv}{\mathbf{v}}
\DeclareMathOperator{\wv}{\mathbf{w}}
\DeclareMathOperator{\xv}{\mathbf{x}}
\DeclareMathOperator{\yv}{\mathbf{y}}
\DeclareMathOperator{\zv}{\mathbf{z}}
\DeclareMathOperator{\wk}{walk}

\DeclareMathOperator{\shad}{shadow}
\DeclareMathOperator{\sh}{sh}
\DeclareMathOperator{\arm}{arm}
\DeclareMathOperator{\leg}{leg}

\DeclareMathOperator{\lin}{lines}
\DeclareMathOperator{\dis}{dist}

\newcommand{\ol}[1]{\overline{#1}}
\newcommand\R{\mathbb{R}}
\newcommand\Z{\mathbb{Z}}
\newcommand\Q{\mathbb{Q}}

\newcommand\N{\mathbb{N}}

\newcommand{\sos}{\text{S\'os}}
\newcommand{\la}{\langle}
\newcommand{\ra}{\rangle}
\newcommand{\lam}{\lambda}
\newcommand{\Lcal}{\mathcal{L} }
\newcommand{\Mcal}{\mathcal{M} }
\newcommand{\Rcal}{\mathcal{R} }
\newcommand{\li}{\chi}
\newcommand{\lf}{\textrm{left}}
\newcommand{\rt}{\textrm{right}}
\newcommand{\tp}{\textrm{top}}
\newcommand{\bo}{\textrm{bot}}

\title[The Schensted shape of a S\'os permutation]{Monotone subsets in lattices and the Schensted shape of a S\'os permutation}

\author[K. Liechty]{Karl Liechty}
\author[T. K. Petersen]{T. Kyle Petersen}
\thanks{KL was supported by Simons Foundation Collaboration Grant \#357872. TKP was supported by Simons Foundation Collaboration Grant \#353772.
Both authors are grateful to Katherine Stange for invaluable discussions. Dan Romik provided helpful references.}
\address{Department of Mathematical Sciences, DePaul University, Chicago, IL}

\date{\today}

\begin{document}

\maketitle

\begin{abstract}
For a fixed irrational number $\alpha$ and $n\in \N$, we look at the shape of the sequence $(f(1),\ldots,f(n))$ after Schensted insertion, where $f(i) = \alpha i \mod 1$.  Our primary result is that the boundary of the Schensted shape is approximated by a piecewise linear function with at most two slopes. This piecewise linear function is explicitly described in terms of the continued fraction expansion for $\alpha$. Our results generalize those of Boyd and Steele, who studied longest monotone subsequences. Our proofs are based on a careful analysis of monotone sets in two-dimensional lattices.
\end{abstract}


\section{Introduction}

Fix a real number $\alpha$, and let $f=f_{\alpha}(i) = \alpha i \mod 1 = \alpha i - \lfloor \alpha i \rfloor$ denote the function that returns the fractional part of multiples of $\alpha$. It suffices to restrict attention to $\alpha\in (0,1)$,  since $f_{\alpha} = f_{\alpha +k}$ for any integer $k$. Now, for any positive integer $n$ we can consider the sequence $(f(1),\ldots,f(n))$. These sequences have been studied quite a bit since the 1950's when Steinhaus made the following conjecture: The points $\{e^{2\pi \mathrm{i} f(j)}\}_{j=1}^n$ divide the unit circle into $(n-1)$ pieces of at most 3 distinct lengths. This conjecture was proved shortly thereafter by \sos\, \cite{Sos}, Sur\'{a}nyi, \cite{Suranyi} and \'{S}wierczkowski \cite{Swierczkowski}, independently of one another. This result, known as the Three Gaps Theorem, has since been proven many times by many different methods. It has found applications in and connections to such disparate places as quantum mechanics \cite{Bleher1, Bleher2}, plant growth \cite{Ravenstein}, combinatorics \cite{AlessandriBerthe}, music theory \cite{Narushima}, Riemannian geometry \cite{BiringerSchmidt}, and of course number theory \cite{Hartman,Sos,vanRavenstein}.

In this paper we study a somewhat coarser object: the permutation induced by the list $(f(1),\ldots,f(n))$. Let $w=w(n,\alpha)$ denote the sorting permutation for this list of numbers, i.e., the lexicographically first permutation such that $f(w(1)) \leq \cdots \leq f(w(n))$. We call such a permutation a \emph{S\'os permutation}, following \cite{BKPT}. (Actually, \cite{BKPT} studies the slightly larger set of permutations generated by fractional parts of lines $f_{\alpha,\beta}(i) = \alpha i + \beta \mod 1$, depending on two independent real parameters. But, as discussed in Remark \ref{rem:betashift} at the end of Section \ref{sec:greene}, it will be enough for our purposes to focus on the $\beta = 0$ case.) Let $\sos_n$ denote the set of all such permutations. That is, for fixed $n$, let 
\[
\sos_n = \{ \pi \in S_n : \pi = w(n,\alpha) \mbox{ for some } \alpha \in (0,1) \}.
\] 
These permutations satisfy their own Three Gap Property, as was noticed by \sos{} \cite{Sos}. See Equation \eqref{eq:sos_structure} below.

Our goal is to characterize the \emph{Schensted shape} of these permutations. 

\subsection{The shape of a permutation}

The phrase ``Schensted shape'' comes from a bijection known as \emph{Schensted insertion} or more generally, the \emph{Robinson-Schensted-Knuth correspondence} (RSK for short). This is a bijection between the set of permutations $S_n$ and ordered pairs of \emph{standard Young tableaux}. Drawn in the so-called French style, these tableaux are lower-left justified arrays of $n$ boxes filled with the integers $1,\ldots,n$ such that the numbers increase across rows and up columns.  If we read the lengths of the rows of boxes, we get a partition of $n$, which we call the \emph{shape} of the tableaux. The precise definition of the correspondence between permutations and pairs of tableaux is not important for now. The important thing to know is that the correspondence gives each permutation a well-defined shape in the form of an integer partition. For the permutation $w$, we write $\sh(w)$ for its shape. A feature of Schensted insertion is that if $w\mapsto (P,Q)$, then $w^{-1}\mapsto (Q,P)$. In particular, $\sh(w) = \sh(w^{-1})$, and when we come to prove our main results it will be convenient to work with $w^{-1}$ rather than $w$.

For example, if $\alpha=.3$ and $n=7$, we have 
\[
 (f(1), f(2), f(3), f(4), f(5), f(6), f(7)) = ( .3, .6, .9, .2, .5, .8, .1),
\]
which when sorted is
\[
 (f(7), f(4), f(1), f(5), f(2), f(6), f(3))=  (.1, .2, .3, .5, .6, .8, .9).
\]
Thus $w = 7415263$ is a permutation in $\sos_7$. The permutation $w^{-1} = 3572461$ is order-isomorphic to the list $(.3, .6, .9, .2, .5, .8, .1)$ prior to sorting. The image of $w$ under Schensted insertion is shown in Figure \ref{fig:RSKex}. This pair of tableaux has shape $\sh(w) = (3,3,1)$, corresponding to the row lengths of the tablueax.  

\begin{figure}
\[
 7415263 \leftrightarrow \left(\Yfrench \young(123,456,7), \young(146,257,3) \right)
\]
\caption{A permutation $w$ corresponds to a pair of Young tableaux under the Schensted insertion map. In this example, the shape of $w$ is $\sh(w)= (3,3,1)$.}\label{fig:RSKex}
\end{figure}

Let $\lam =(\lam_1, \lam_2, \dots, \lam_m)$ be an integer partition with $\lambda_1\geq \lambda_2 \geq \cdots \geq \lambda_m > 0$ and $\sum \lambda_i = n$. Following Romik \cite{Romik}, we define the \emph{planar set} of $\lam$ to be the collection of $n$ boxes given by
\[
A_\lam = \bigcup_{\substack{1\le k \le m, \\ 1\le l \le \lam_k}} [l-1,l]\times [k-1,k],
\]
which sits in the first quadrant.  We will use the notation $\partial A_\lam$ to denote the piece of the boundary of $A_\lam$ which is strictly in the first quadrant, i.e., the part of the boundary which is not on one of the coordinate axes. The length of the $i$th row of $A_\lambda$, read from bottom to top, is $\lam_i$. We also let $\lam'=(\lam'_1, \ldots,\lam'_{m'})$ denote the conjugate partition, i.e., $\lam_i'$ is the height of the $i$th column of $A_\lam$, read from left to right. This planar set is our concrete realization of the partition $\lambda=\sh(w)$. 

Schensted defined his map to study the longest monotone subsequences of permutations (or data strings generally). In particular, \cite{Schensted} shows the \emph{arm}, $\arm(w)=\lam_1$, is the length of the longest increasing subsequence of $w$, and the \emph{leg}, $\leg(w)=\lam'_1$, is the length of the longest decreasing subsequence of $w$. 

As later shown by Greene \cite{Greene}, Schensted's map allows the following characterization of collections of monotone subsequences.  Let $I_k$ denote the size of the largest subsequence formed by the union of $k$ increasing subsequences of $w$. Equivalently, $I_k$ is the size of the largest subsequence containing no decreasing subsequence of length $k+1$. In an analogous fashion, let $D_k$ denote the size of the largest subsequence obtained as a union of $k$ decreasing subsequences of $w$.

\begin{thm}[Greene]\label{thm:greene}
For each $k\leq n$, 
\[
I_k = \lambda_1+\cdots+\lambda_k \quad \mbox{ and } \quad D_k = \lambda_1' + \cdots + \lambda_k'.
\]
\end{thm}

Thus, the partition $\sh(w)$ contains all the information one needs about monotone subsequences in $w$, and a ``limit shape theorem'' for planar sets $A_\lam$ can yield detailed information about the expected number and size of monotone subsequences in random permutations.

The literature abounds with examples of shape theorems for random permutations and related objects, many of which are catalogued and explained in Romik's book \cite{Romik}. The most well-known example is the now famous result obtained independently by Vershik and Kerov \cite{VershikKerov} and Logan and Shepp \cite{LoganShepp} that describes a precise limit shape for a uniformly random permutation $w \in S_n$ under the Schensted insertion map. See equation \eqref{eq:vershikkerov} below. The related question we would like to answer is the following.

\begin{ques}
Given a uniformly random S\'os permutation, what is its expected shape?
\end{ques}

At the moment, we lack the understanding to adequately address this question, though this paper provides some first steps. One issue is that ``uniform'' could mean two different things here. It could mean:
\begin{itemize}
\item \emph{combinatorially uniform}, i.e., for each $n$ we select an element of $\sos_n$ with probability $1/|\sos_n|$, or
\item \emph{geometrically uniform}, i.e., for each $n$ we select $\alpha$ uniformly in $(0,1)$ and generate $w(n,\alpha)$.
\end{itemize}

As explained in \cite{BKPT}, a result of Sur\'anyi shows that S\'os permutations are in bijection with \emph{Farey intervals}. For fixed $n$, a Farey interval is of the form $(a/b,c/d)$, $0\le a/b < c/d \le 1$, where $a/b$ and $c/d$ are consecutive reduced fractions of denominator at most $n$. Sur\'anyi's result says the following.

\begin{thm}[Sur\'anyi's bijection]\label{thm:suranyi}
There is a bijection between Farey intervals and \sos\, permutations. That is, $w(n,\alpha)=w(n,\alpha')$ if and only if $\alpha$ and $\alpha'$ lie in the same Farey interval.
\end{thm}

Thus, by counting Farey intervals we see the number of \sos{} permutations is 
\[
|\sos_n| = \sum_{k\leq n} \varphi(k) \underset{n\to\infty}{\sim}\frac{3n^2}{\pi^2},
\] 
where $\varphi$ is the Euler totient function (see, e.g., \cite[Section 18.5]{HardyWright}). We see the set $\sos_n$ represents a very small subset of  full symmetric group $S_n$, and it is not surprising that the permutations in $\sos_n$ are highly structured. Indeed, it was shown by \sos\, \cite[Theorem 1]{Sos} that for any $\pi\in \sos_n$ corresponding to the Farey interval $(a/b,c/d)$, we have $\pi(1)=b$ and
 \eq\label{eq:sos_structure}
 \pi(i+1) - \pi(i) = \left\{
 \begin{aligned}
 &b \qquad &\textrm{if} \ \pi(i) &\le n-b \\
   &b-d \qquad &\textrm{if} \ n-b&<\pi(i) < d \\
      &-d \qquad &\textrm{if} \ d&\le\pi(i). \\
      \end{aligned}\right.
 \eeq

Further, Sur\'anyi's result shows how to compute the geometrically uniform distribution on $\sos_n$. If $\pi = w(n,\alpha)$ with $a/b \leq \alpha < c/d$, then the geometric probability of $\pi$ is  $c/d-a/b = 1/bd$. In particular, the identity permutation has probability $1/n$ in the geometrically uniform distribution. 

As the set $\sos_n$ is so highly structured, it may come as little surprise that the distribution of shapes appears quite different from the case of the full symmetric group. Indeed, as we will see stated precisely in Theorem \ref{thm:main}, our best answer to the Big Question (for now) is:

\begin{quote}
\textit{The Schensted shape of a \sos{} permutation has a boundary that is approximately piecewise linear, with at most two slopes.}
\end{quote}

Before making that claim more precise, we review work of Boyd and Steele \cite{BoydSteele} and del Junco and Steele \cite{dJS}, who studied the lengths of the longest increasing and decreasing subsequences in  $w(n,\alpha)$ for $\alpha$ fixed and $n$ varying. 

\subsection{Arms and legs for $w(n,\alpha)$}

Results concerning monotone subsequences from finite strings of real numbers date back at least to the work of Erd\H os and Szekeres \cite{ErdosSzekeres35}, who proved that any sequence of $n^2+1$ distinct real numbers has a monotone subsequence of length $n+1$. For uniformly random permutations in all of $S_n$, the Logan--Shepp/Vershik--Kerov limit shape theorem implies that both the longest increasing subsequence, $\arm(w)$, and longest decreasing subsequence, $\leg(w)$, are asymptotic to $2\sqrt{n}$ as $n\to\infty$. Actually the problem of the asymptotic length of the longest increasing subsequence in a random permutation, known as Ulam's problem, is a fascinating story which predates \cite{LoganShepp, VershikKerov} by about 15 years. We refer interested readers to the book \cite{Romik} and the survey article \cite{Steele95}. 

In contrast with the behavior of uniformly random permutations in $S_n$, the papers \cite{BoydSteele, dJS} found more subtle behavior for the \sos\ permutations $w(n,\alpha)$. In fact, while the paper \cite{dJS} found
\[
 \frac{\log(\arm(w(n,\alpha)))}{\log(n)} \to \frac{1}{2} \quad \mbox{ and } \quad \frac{\log(\leg(w(n,\alpha)))}{\log(n)} \to \frac{1}{2}
\]
for almost all $\alpha$, the paper \cite{BoydSteele} found there is no $\alpha$ for which either $\arm(w(n,\alpha))/\sqrt{n}$ or $\leg(w(n,\alpha))/\sqrt{n}$ tends to a limit! 

The results of Boyd and Steele \cite{BoydSteele} imply that as a function of $n$, the quantity $\arm(w(n,\alpha))/\sqrt{n}$ oscillates between local maxima and minima that can be precisely described in terms of the principal convergents in the continued fraction expansion of $\alpha$, and as a consequence $\lim \sup \arm(w(n,\alpha))/\sqrt{n}$ is finite if and only if the terms of the continued fraction are bounded. In this paper, we will reprove the results of Boyd and Steele in the process of extending them to a tight approximation for the Schensted shape of $w(n,\alpha)$. Before we present the precise version of this approximation, it is helpful to understand Boyd and Steele's results in more detail. 

We first introduce some notation for continued fractions. (See \cite{HardyWright} for a general reference on the topic.) For a real number $x$, its (simple) continued fraction expansion is defined recursively via
\[
 a_0 = \lfloor x \rfloor, \quad x^*_0 = \frac{1}{x-a_0},
\]
and
\[
 a_k = \lfloor x^*_{k-1} \rfloor, \quad x^*_k = \frac{1}{x^*_{k-1}-a_{k-1}},
\]
provided $x^*_{k-1}$ is not an integer. 
We write
\[
 x = [a_0; a_1, a_2, \ldots] = a_0 + \frac{1}{a_1 + \frac{1}{a_2 + \frac{1}{\ddots}}}.
\]
If for some $k$, $x^*_{k-1}$ is an integer, then $a_k$ is the final term in the sequence, and $x$ is rational.

For a fixed real number $\alpha = [a_0; a_1, a_2,\ldots]$, we have a sequence of rational numbers 
\eq\label{def:convergents}
 \frac{p_k}{q_k} = [a_0; a_1, a_2 \ldots, a_k],
\eeq
known as the \emph{principal convergents} to $\alpha$. We denote by $\delta_k = |\alpha - p_k/q_k|$ the distance between $\alpha$ and its $k$th principal convergent, and let $\beta_k = \delta_k^{-1}$ denote the reciprocal of this distance. A fundamental result about approximating with convergents is the approximation (see, e.g., \cite[Theorem 164]{HardyWright})
\begin{equation}\label{eq:diophantine}
 \delta_k < \frac{1}{q_{k+1}q_k}< \frac{1}{q_k^2}.
\end{equation}
We can now give precise statements of the main results of Boyd and Steele. 

\begin{thm}[Arm and leg approximations, see Theorem 1 and Corollary 1 of \cite{BoydSteele}]\label{thm:BSarm}
For $\beta_{2h} \leq n < \beta_{2h+2}$,
\[
 q_{2h+1}(1+n\delta_{2h+1}) -2 < \arm(w(n,\alpha))   \le  q_{2h+1}(1+n\delta_{2h+1})
\]
while for $\beta_{2k-1}\leq n < \beta_{2k+1}$,
\[
 q_{2h}(1+n\delta_{2h}) - 2 < \leg(w(n,\alpha))\le  q_{2h}(1+n\delta_{2h}).
\]
\end{thm}

These bounds follow naturally from the approach we take in this paper, and a proof is presented in Sections \ref{sec:my} (see Proposition \ref{prp:l+bounds}) and \ref{sec:proof}. Our proof is rather similar to the one in \cite{BoydSteele}, with the main difference being that we describe the integer programming problem presented in \cite[Lemma 3]{BoydSteele} as a problem of longest increasing paths on lattices. Our description lends itself nicely to a simultaneous treatment of increasing and decreasing subsequences which is necessary to understand the full Schensted shape.

Dividing the expressions in Theorem \ref{thm:BSarm} by $\sqrt{n}$ and computing critical points with respect to $n$, it is easy to obtain the following result.

\begin{thm}[Normalized arm and leg extrema, see Theorem 2 of \cite{BoydSteele}]\label{thm:BS}
We have the following local extremes for the normalized arm length $\arm(w(n,\alpha))/\sqrt{n}$ and the normalized leg length $\leg(w(n,\alpha))/\sqrt{n}$.
\begin{enumerate}
\item When $n$ is the nearest integer to $\beta_{2h}$, the normalized arm length achieves a local maximum of approximately 
\[
 M^+_{2h}=\frac{1}{q_{2h}\delta_{2h}^{1/2}}
\]
and the normalized leg length achieves a local minimum of approximately
\[
 m^-_{2h} = 2q_{2h}\delta_{2h}^{1/2}.
\]
\item When $n$ is the nearest integer to $\beta_{2h+1}$, the normalized arm length achieves a local minimum of approximately 
\[
 m^+_{2h+1}=2q_{2h+1}\delta_{2h+1}^{1/2}
\]
and the normalized leg length achieves a local maximum of approximately
\[
 M^-_{2h+1} = \frac{1}{q_{2h+1}\delta_{2h+1}^{1/2}}.
\]
\end{enumerate}
\end{thm}

We use the example of $\alpha = e = [2; 1, 2, 1, 1, 4, 1, 1, 6,\ldots]$ to illustrate here. In Figure \ref{fig:armleg} we plot $\arm(w(n,e))/\sqrt{n}$ and $\leg(w(n,e))/\sqrt{n}$ to give the reader a feel for the behavior of these sequences. We can see the resemblance of the continued fraction itself reflected in the heights of the peaks and valleys. Note the horizontal scale is logarithmic.

\begin{figure}
\includegraphics[width=14cm]{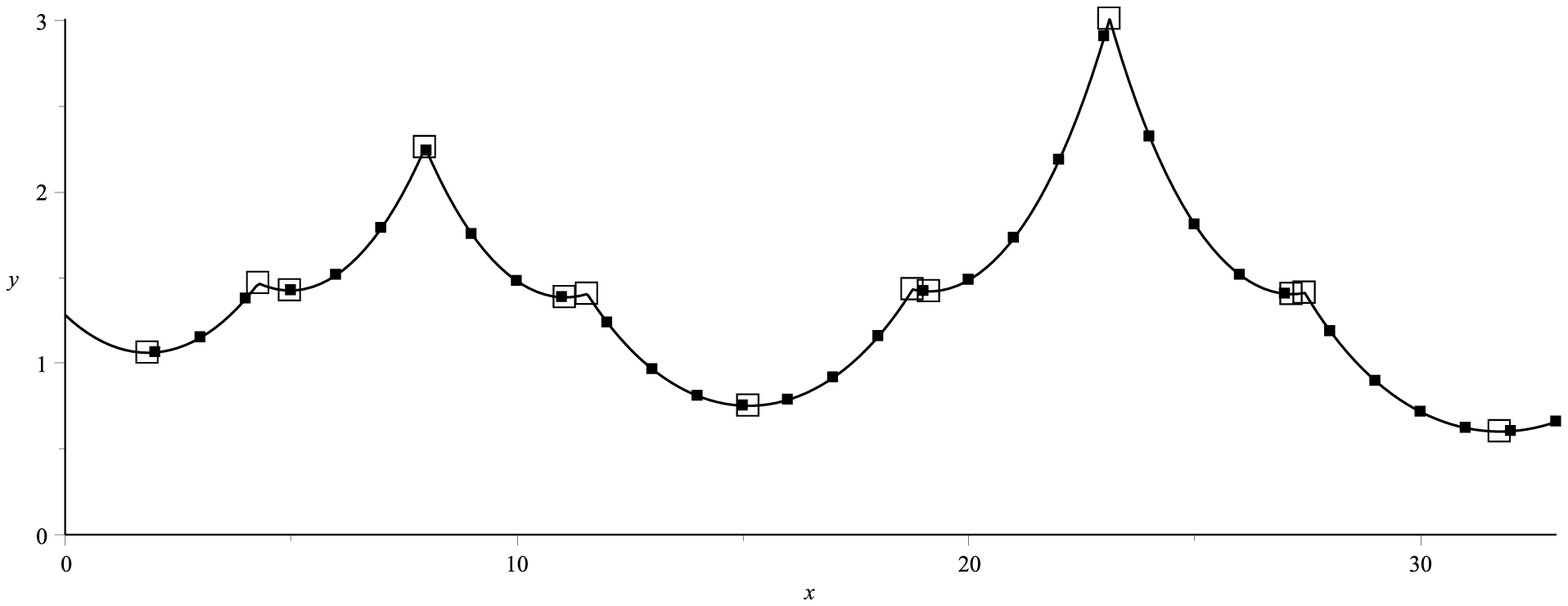}

\includegraphics[width=14cm]{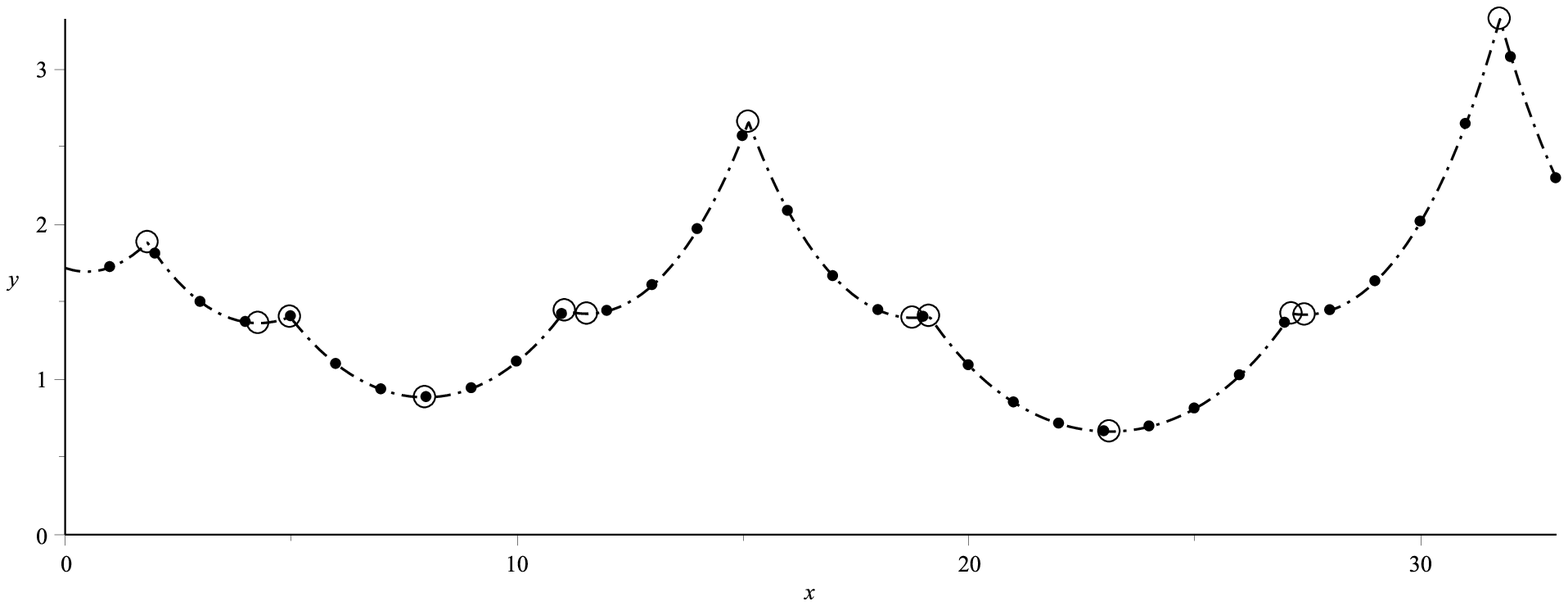}
\caption{Plot of normalized arm and leg lengths for $w(2^n,e)$ for $n\leq 33$. Solid squares mark the points $(n,\arm(w(2^n,e))/\sqrt{2^n})$, while open squares highlight local maxima/minima on the curve. The solid curve predicts the arm length. Solid circles mark the points $(n,\leg(w(2^n,e))/\sqrt{2^n})$, while open circles highlight local maxima/minima on the curve. The dashed curve predicts the leg length.}\label{fig:armleg}
\end{figure}

\subsection{Shape of $w(n,\alpha)$}

When the arm and leg lengths are close to their local extreme values, one can predict predict the Schensted shape of $w(n,\alpha)$ using Theorem \ref{thm:BS} above. A simple situation is when $n$ is approximately equal to $\beta_{2h}$. Here, Theorem \ref{thm:BS} tells us the normalized arm is approximately $M_{2h}^+$ and the normalized leg is approximately $m_{2h}^-$. Since, according to Theorem \ref{thm:BS}, their product is $M_{2h}^+m_{2h}^- = 2$, it is reasonable to predict that the normalized shape is approximated by a triangle with $x$-intercept $M_{2h}^+$ and $y$-intercept $m_{2h}^-$. The boundary of the shape is thus approximated by the line
\eq\label{eq:triangle1}
  y = m_{2h}^- - \frac{m_{2h}^-}{M_{2h}^+}x =  2q_{2h}\delta_{2h}^{1/2} - 2q_{2h}^2\delta_{2h}x.
\eeq
The case when $n$ is approximately $\beta_{2h+1}$ is similar. In this case we would guess the shape is a triangle whose boundary is described by
\eq\label{eq:triangle2}
 y = M_{2h+1}^- - \frac{M_{2h+1}^-}{m_{2h+1}^+} x = \frac{1}{q_{2h+1}\delta_{2h+1}^{1/2}} - \frac{1}{2q^2_{2h+1} \delta_{2h+1}} x.
\eeq
These predictions are verified in Theorem \ref{thm:main} below.

When $n$ is between $\beta_{2h}$ and $\beta_{2h+1}$, we find the shape of $w(n,\alpha)$ is only slightly more complicated: the boundary is approximated by a piecewise linear function whose two linear components are parallel to the lines \eqref{eq:triangle1} and \eqref{eq:triangle2}, and with $x$- and $y$-intercepts given by the arm and leg measurements predicted by Theorem \ref{thm:BSarm}. This is our main result, which we have split into two parts.

\begin{thm}[Estimating $\lambda$ and $\lambda'$]\label{thm:lam_bounds} Suppose $\alpha\in (0,1)$ is an irrational number with principal convergents $\{p_i/q_i\}$. Fix $n\in \N$ large enough so that $n> 1/\alpha$ and $n> 1/(1-\alpha)$. Denote $\lam = \sh(w(n,\alpha))$, the shape of the permutation $w(n,\alpha)$.  With $h$ such that $\beta_{2h} \leq n <\beta_{2h+2}$, define the values $x_0$ and $y_0$ to be
\[
x_0:=\frac{1-n\delta_{2h+1}}{q_{2h+*}(\delta_{2h+1}-\delta_{2h+*})}, \qquad y_0:=\frac{1-n\delta_{2h+*}}{q_{2h+1}(\delta_{2h+*}-\delta_{2h+1})},
\]
where 
\[
*=
\begin{cases} 
0 & \mbox{ if $\beta_{2h} \leq n < \beta_{2h+1}$,}\\
2 & \mbox{ if $\beta_{2h+1} \leq n < \beta_{2h+2}$.}
 \end{cases}
\]
For all $1\leq k \leq y_0-1$, we have
\[
|\lambda_k - q_{2h+1}(1+n\delta_{2h+1}) + 2kq_{2h+1}^2\delta_{2h+1}| < 4+2q^2_{2h+1}\delta_{2h+1}<6,
\]
and for all $1\leq k \leq x_0-1$, we have
\[
|\lambda_k' - q_{2h+*}(1+n\delta_{2h+*}) + 2kq_{2h+*}^2\delta_{2h+*}| < 4+2q_{2h+*}^2\delta_{2h+*}<6.
\]
\end{thm}

\begin{remark} 
In Theorem \ref{thm:lam_bounds} above and Theorem \ref{thm:main} below, the technical assumption that $n>1/\alpha$ and $n>1/(1-\alpha)$ simply ensures that $w(n,\alpha)$ is not the identity permutation or its reverse. In these cases, the Schensted shape of $w(n,\alpha)$ is trivial, consisting of a single row or column of length $n$.
\end{remark}

The above theorem indicates that the boundary of the planar set $A_\lambda$ is very close to one of the two negatively sloped lines
\[
  y = q_{2h+*}(1+n\delta_{2h+*}) - 2q_{2h+*}^2\delta_{2h+*}x \quad \mbox{ and } \quad  x = q_{2h+1}(1+n\delta_{2h+1}) - 2q_{2h+1}^2\delta_{2h+1}y.
\]
 It turns out that these two lines intersect exactly at the point $(x_0,y_0)$ (see equations \eqref{eq:xy-lines} and \eqref{eq:x0y0} in Section \ref{sec:shape-bounds}) and so the boundary of $A_\lambda$ is uniformly close to the piecewise linear function
\begin{equation}\label{eq:Lnalpha}
 L(n,\alpha; x) =  
 \begin{cases} 
 q_{2h+*}(1+n\delta_{2h+*}) - 2q_{2h+*}^2\delta_{2h+*}x & \mbox{for $0\leq x \leq x_0$,} \\
\frac{q_{2h+1}(1+n\delta_{2h+1})-x}{2q_{2h+1}^2\delta_{2h+1}} & \mbox{for $x_0 \leq x \leq q_{2h+1}(1+n\delta_{2h+1})$}.
\end{cases}
\end{equation}
We capture this idea in our next main result, where for fixed $x$, $\dis( L(n,\alpha;x), \partial A_{\lambda} )$ denotes the minimum Euclidean distance from $(x,L(x))$ to a point on $\partial A_{\lambda}$.

\begin{thm}[Two Slope Theorem]\label{thm:main}
For a fixed irrational number $\alpha\in (0,1)$ and $n\in \N$ satisfying  $n>1/\alpha$ and $n>1/(1-\alpha)$, let $\lambda$ and $h$ be as in Theorem \ref{thm:lam_bounds}, and let $L(n,\alpha;x)$ be the piecewise linear function defined in \eqref{eq:Lnalpha}. 
We have the uniform estimate
\[
\dis\left(L(n,\alpha;x), \partial A_\lam\right)< 8, \quad \textrm{for all} \ 0 \le x \le q_{2h+1}(1+n\delta_{2h+1}).
\]
\end{thm}

For example if $\alpha = e$, let us take $n=4700$. We have \[
\delta_6 = e-\frac{p_6}{q_6} = .0003331\ldots > \frac{1}{4700} = .0002127\ldots > .00002803\ldots = \frac{p_7}{q_7}-e = \delta_7,\] 
so $\beta_6 < 4700 < \beta_7$. We compute the slopes as
\[
-2q_{6}^2\delta_{6} = -1.01332\ldots \quad \mbox{ and } \quad -\frac{1}{2q_{7}^2\delta_{7}} = -3.53850\ldots.
\]
It is useful to contrast the shape described in Theorem \ref{thm:main} with the Logan--Shepp/Vershik--Kerov limit shape for a uniformly random permutation in $S_n$. They showed that with high probability, after rescaling by a factor of $1/\sqrt{n}$, the boundary of the corresponding Young diagram for a uniformly random permutation is approximated very well by the curve
\eq\label{eq:vershikkerov}
x+y = \frac{2}{\pi}\left((x-y)\sin^{-1}\left(\frac{x-y}{2}\right)+\sqrt{4-(x-y)^2}\right), \qquad 0\le x\le 2, \ 0\le y \le 2.
\eeq
This curve lies in stark contrast to the piecewise linear boundary for a \sos\, permutation described in Theorem \ref{thm:main}. See Figure \ref{fig:epicture}.

\begin{remark}
The uniform bounds presented in Theorems \ref{thm:lam_bounds} and \ref{thm:main} are not optimal.  They can be improved, even by the methods of this paper. Numerical examples such as Figure \ref{fig:epicture} indicate a very tight bound which seems to be less than 2.  Since we were not able to prove a bound close to this, we focused on readability rather than obtaining the sharpest bounds possible, reasoning that not much is gained by a marginally better bound which is still not optimal. 
\end{remark}

\begin{figure}
\includegraphics[height=10cm]{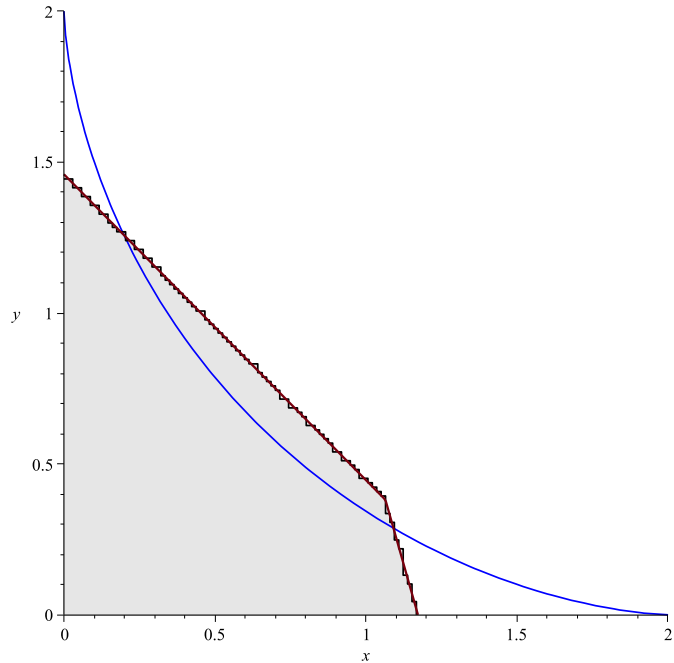}
\vspace{-10pt}
\caption{The planar set of $w(4700,e)$, with boundary approximated by $L(4700,e;x)$ in red. Here the axes are scaled by a factor or $1/\sqrt{n}$, and for contrast we have superimposed the Logan--Shepp/Vershik--Kerov limit shape for a uniformly random permutation in blue.} \label{fig:epicture}
\end{figure}

\subsection{Plan for the rest of the paper} We will prove our main theorems (Theorems \ref{thm:BSarm}, \ref{thm:lam_bounds}, and \ref{thm:main}) as follows. In Section \ref{sec:paths}, we will translate our problem of monotone subsequences in a \sos{} permutation to the problem of monotone paths in an integer lattice. Next, in Section \ref{sec:fulllattice}, we will characterize the shape of permutations given by the lattice $\{(i, a\cdot i \mod N) : i=1,\ldots,n \}$, for any integers $0<a<N$ and $0<n<N$ with $a$ and $N$ relatively prime. We show in Section \ref{sec:proof} how this case of rational numbers implies our results for any irrational $\alpha$, which will complete the proofs.

Finally, in Section \ref{sec:further}, we describe a few further directions this research could take.

\section{Increasing and decreasing lattice paths}\label{sec:paths}

In this section we investigate the geometric interpretation of monotone subsequences of a permutation $w$ via \emph{increasing subsets} of the points $\{(i,w(i))\}$. This is an old idea, which (as explained in \cite{Romik}) goes back to Hammersley \cite{Hammersley72}. For \sos{} permutations we have the special advantage that the set $\{(i, f_{\alpha}(i) )\}$ is a lattice, which, up to vertical scaling, is the same as $\{(i, w^{-1}(i))\}$. As mentioned, $\sh(w) = \sh(w^{-1})$, so it suffices to study increasing subsets in this set of points. The remainder of this section lays out all the details of our approach to increasing subsets in the context of lattices.

\subsection{Lattices and monotone lattice walks}

In general, a \emph{two-dimensional integer lattice} $L$ is the set of all integer linear combinations of two linearly independent vectors $\av$ and $\bv$ in $\R^2$, denoted $L=\Z[\av, \bv] = \{ k\av + l\bv: k, l \in \Z\}$. A \emph{lattice vector} refers to the displacement vector between any two points in $L$. We write points in $L$ as ordered pairs $(a,b)$, while vectors are given in angle brackets, e.g., $\la a, b \ra$. By abuse of notation we use boldface letters for both points and vectors. The meaning will usually be clear from context, and we will alert the reader when care must be taken.

Given a vector $\xv=\la x_1, x_2\ra$, we denote the slope of the vector by 
\[
m_{\xv}:=x_2/x_1,
\]
provided $x_1 \neq 0$. (If $x_1 = 0$ we write $m_{\xv} = \pm\infty$ according to whether the second component $x_2$ is positive or negative.) We will call a non-zero vector $\xv$ an {\it increasing vector} if both components of $\xv$ are non-negative; we will call $\xv$ a {\it decreasing vector} if $x_1\ge 0$ and $x_2 \le 0$. Note that under this definition, the vector $\la x_1, 0\ra$ is both increasing and decreasing whenever $x_1>0$.
A sequence of increasing lattice vectors $\xv_1, \ldots, \xv_k$ is called an \emph{increasing walk}, and a sequence of decreasing lattice vectors $\yv_1, \ldots, \yv_j$ is called an \emph{decreasing walk}. If we think of each vector in the sequence as a step on the lattice beginning at the origin, then the increasing walk $\xv_1, \ldots, \xv_k$ begins at $(0,0)$ and ends at $\xv = \sum_{i=1}^k \xv_i$ in the first quadrant, while the decreasing walk $\yv_1, \ldots, \yv_j$ begins at $(0,0)$ and ends at $\yv = \sum_{i=1}^j \yv_i$ in the fourth quadrant.

Conversely, given any lattice point $\xv$ in the first quadrant, there are a number of increasing walks from $(0,0)$ to $\xv$. Denote the set of such walks by
\[
 \wk^+(\xv) = \left\{ (\xv_1,\ldots,\xv_k) :{\xv_i}  \mbox{ are all increasing vectors and } \sum_{i=1}^k \xv_i = \xv\right\}.
\]
Similarly, for any lattice point $\yv$ in the fourth quadrant we define the collection of decreasing walks from $(0,0)$ to $\yv$ taking negatively sloped steps:
\[
 \wk^-(\yv) = \left\{ (\yv_1,\ldots,\yv_j) : \yv_i  \mbox{ are all decreasing vectors and } \sum_{i=1}^j \yv_i = \yv\right\}.
\]

The \emph{increasing lattice length} of a vector $\xv$ in $L$ is the greatest number of steps in an increasing walk to $\xv$:
\[
 \ell^+(\xv) = \max \{ k : (\xv_1,\ldots,\xv_k) \in \wk^+(\xv)\}.
\]
We similarly define the \emph{decreasing lattice length} of $\yv$:
\[
 \ell^-(\yv) = \max \{ k : (\yv_1,\ldots,\yv_k) \in \wk^-(\yv)\}.
\]
Globally, for a finite subset $S$ of a lattice, we denote the maxima of these quantities by 
\[
 \ell^+=\ell^+(S) = \max \{ \ell^+(\xv) : \xv \in S\} \quad \textrm{ and } \quad \ell^-=\ell^-(S) = \max \{ \ell^-(\yv) : \yv \in S\}.
\]

It will be important to understand lattice length in terms of vectors with lattice length one. A helpful way to characterize such vectors is in terms of their \emph{shadow}. For a lattice vector $\xv = \la x, y\ra$, define the set $\shad(\xv)$ to be the set of nonzero lattice points in the rectangle spanned by $\la x, 0\ra$ and $\la 0, y\ra$, not including $\xv$ itself. The key feature of shadows is the following.

\begin{obs}[Empty shadows]\label{obs:suff1}
Let $\xv \in L$. We have $\shad(\xv) = \emptyset$ if and only if $\xv$ has lattice length one.
\end{obs}

\begin{proof}
For simplicity, assume $\xv$ is an increasing lattice vector. Every sequence in $\wk^+(\xv)$ necessarily stays in the shadow of $\xv$. Therefore, if no lattice points are inside the shadow, the lattice length must be one. 

Conversely, if there is a lattice point in the shadow, call the vector from the origin to that point $\xv'$. Note also that $\xv-\xv'$ is an increasing lattice vector, and thus the path $( \xv', \xv-\xv')$ shows $\ell^+(\xv) >1$.
\end{proof}

From now on, we will assume that $L$ is a finite lattice on a torus, given by all multiples of an integer $a$ modulo $b$, with $a < b$ and $\gcd(a,b)=1$. That is, we now declare
\[
L=L_{a,b} = \{ (i, a\cdot i ) \mod b : i \in \Z \} = L[ \la 1, a\ra, \la -1, b-a \ra ] \pmod b
\]
We can picture this lattice as lying in the square $[0,b]\times [0,b]$. We can interpret any point $(x,y)$ in the square as an increasing vector $\xv = \la x,y \ra$ or as a decreasing vector $\yv = \xv - \la 0,b\ra = \la x, y-b\ra$. Hence to each point we can assign an increasing length, $\ell^+(\xv)$, and a decreasing length, $\ell^-(\yv)$. This is illustrated in Figure \ref{fig:ablattice} for $a = 51$ and $b=71$. The marked point is $(30,39)$, so that $\xv = \la 30, 39\ra$ and $\yv = \la 30,-32\ra$.

\begin{figure}
\[
\begin{tikzpicture}[scale=.12, >=stealth]
\foreach \x in {0,...,71}{
 \draw[fill=black] (\x, {mod(\x*51,71)} ) circle (5pt);
}
\draw[fill=black] (0, 71) circle (5pt);
\draw[fill=black] (71, 71) circle (5pt);
\draw (0,71)--(0,0)--(71,0)--(71,71)--(0,71);
\draw[line width =2, ->, opacity=.5, color=blue] (0,0)--(37,41 ) node[midway, left, opacity=1, color=black] {$\xv$};
\draw[->] (0,71)--(4,62);
\draw[->] (4,62)--(37,41);
\draw[line width =2, ->, opacity=.5,color=red] (0,71)--(37, 41 ) node[midway, below left, color=black,opacity=1] {$\yv$};
\draw[->] (33,50)--(37,41) node[midway,right,fill=white, inner sep=1] {$\vv_2$};
\draw[->] (0,71)--(33,50) node[midway,above right] {$3\vv_3$};
\draw[->] (0,0)--(28,8) node[midway,below right,fill=white, inner sep=1] {$4\uv_4$};
\draw[->] (28,8)--(37,41) node[midway,right] {$3\uv_3$};
\draw[->] (0,0)--(9,33);
\draw[->] (9,33)--(37,41);
\end{tikzpicture}
\]
\caption{The lattice $L_{51,71}$. The vector $\xv=\la 30,39\ra$ has $\ell^+(\xv) = 7$, while the vector $\yv = \xv - \la 0,71\ra$ has $\ell^-(\yv) = 4$.}\label{fig:ablattice}
\end{figure}
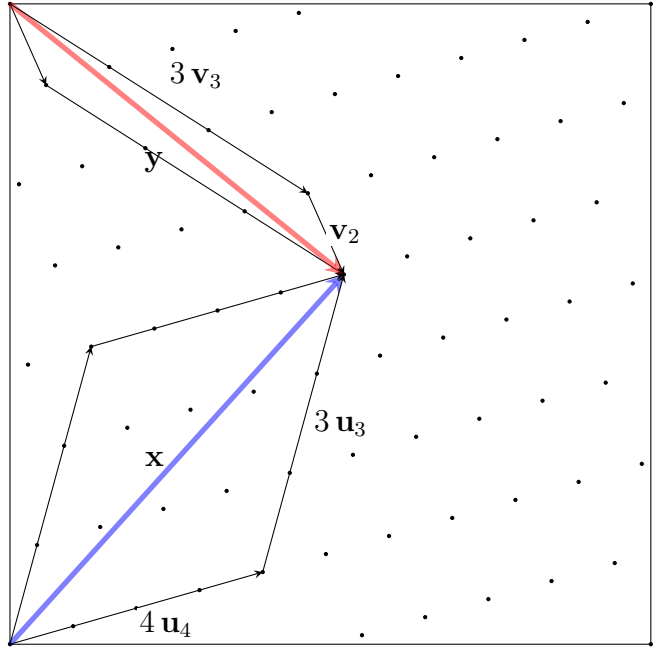

We will soon show that vectors of lattice length one possess many nice properties and will aid in characterizing lattice length for any point in $L$. The first thing to show is how these vectors come from the Euclidean algorithm.

\subsection{The slow Euclidean algorithm and continued fractions}\label{sec:slowEuclid}

For input integers $a$ and $b$, the usual Euclidean algorithm produces a sequence of remainders $r_i$ and integers $s_i$ and $t_i$ such that $r_i = s_i b + t_i a$. We now describe a version of the extended Euclidean algorithm that records not only every stage of the usual Euclidean algorithm, which invokes the division algorithm in each step, but also the steps of subtraction that the division summarizes. We will refer to this algorithm as the ``slow'' Euclidean algorithm.

Initialize the algorithm with two integers $a$ and $b$ such that $1\leq a \leq b$. We will index our steps with pairs $(i,j)$, where at each step, $(i,j)$, we have an integer combination
\eq\label{eq:EuclidanAlg}
 r_{i,j} = s_{i,j}b + t_{i,j}a.
\eeq
We initialize
\[
\begin{array}{lll}
r_{-1,1} = r_{-1}= b, & s_{-1,1}=s_{-1}=1, & t_{-1,1}=t_{-1}=0\\
r_{0,1} = r_0 = a, & s_{0,1}=s_0 = 0, & t_{0,1} = t_0 = 1.
\end{array}
\]

For each $i\geq 1$, rather than divide $r_{i-1}$ by $r_{i-2}$ (as in the common version of the algorithm), we successively subtract $r_{i-1}$ from $r_{i-2}$ until we can no longer subtract and remain nonnegative. That is, we set $r_{i,1} = r_{i-2}-r_{i-1}$, and while $r_{i,j-1}\geq r_{i-1}$ we set 
\[
r_{i,j} = r_{i,j-1}-r_{i-1}=r_{i-2}-jr_{i-1}.
\] 
We similarly define the coefficients $s_{i,j}$ and $t_{i,j}$ by setting 
\begin{equation}\label{eq:strec}
s_{i,j} = s_{i-2}-js_{i-1} \quad \mbox{ and } \quad t_{i,j} = t_{i-2}-jt_{i-1},
\end{equation}
where $j$ ranges over the same values as the $r_{i,j}$. We iterate $i$ when we reach a value of $j'$ for which $r_{i-1} > r_{i,j'} \geq 0$. At this point we set $r_i = r_{i,j'}$, $s_i=s_{i,j'}$, and $t_i = t_{i,j'}$.  We will call these terms with a single subscript the ``simple" remainders and $s, t$-coefficients, since they are the ones that appear in the usual (fast) Euclidean algorithm (with division rather than subtraction at each step). The terms with two subscripts ($r_{i,j}, s_{i,j}, t_{i,j}$) we refer to as ``slow'' remainders and coefficients. We also remark that this $j'$ is the usual integer quotient,  since $r_{i-2} = r_{i-1}j' + r_i$ with $0\leq r_i < r_{i-1}$. The algorithm terminates when we find $r_{i,j} = 0$. 

The Euclidean algorithm has a well known connection to continued fraction expansions. Recalling the language of continued fractions and their convergents from \eqref{def:convergents}, we write the rational number $a/b$ as
\[
 \frac{a}{b} = [a_0; a_1,a_2,\ldots,a_k] = a_0 + \frac{1}{a_1 + \frac{1}{a_2 + \frac{1}{\ddots + \frac{1}{a_k}}}}.
\]
For each $i=0,\ldots,k$, we thus have sequences of relatively prime pairs $(p_i,q_i)$ such that
\[
 \frac{p_i}{q_i} = [a_0; a_1,\ldots,a_i],
\]
with these fractions giving the principal convergents to $a/b$. It follows from the definition that 
\[
p_{i} =  p_{i-2}+ a_{i}p_{i-1}  \quad \mbox{ and } \quad q_{i} = q_{i-2}+ a_{i}q_{i-1}.
\]
Moreover, we can define intermediate convergents, for $j=1,\ldots,a_i$, via
\begin{equation}\label{eq:pqrec}
 p_{i,j} =  p_{i-2}+ jp_{i-1} \quad \mbox{ and } \quad q_{i,j} = q_{i-2} + jq_{i-1},
\end{equation}
with the property that
\[
 \frac{p_{i,j}}{q_{i,j}} = [a_0; a_1,\ldots,a_{i-1},j]. 
\]
We can see the recurrences in \eqref{eq:pqrec} for the $p_{i,j}$ and $q_{i,j}$ are identical, up to sign, to the recurrences in \eqref{eq:strec} for the $s_{i,j}$ and $t_{i,j}$.
Furthermore, we have 
\[
 \delta_i = | a/b - p_i/q_i| \quad \mbox{ and } \quad \delta_{i,j} = |a/b - p_{i,j}/q_{i,j}|,
\]
or equivalently,
\[
 bq_i \delta_i = | q_ia - p_i b| \quad \mbox{ and } \quad bq_{i,j}\delta_{i,j} = |q_{i,j} a - p_{i,j} b|.
\]
Up to sign, these identities are equivalent to that of Equation \eqref{eq:EuclidanAlg} expressing the $r_{i,j}$ in terms of $s_{i,j}$ and $t_{i,j}$.

As we have the same linear recurrence relations (up to sign) with the same initial values, we make the following observation relating the data in the slow Euclidean algorithm for the pair $(a,b)$ with the data in the sequence of convergents of the continued fraction for $a/b$.

\begin{obs}[Slow Euclidean algorithm and continued fraction convergents]\label{obs:SlowCF}
Let $(a,b)$ be a relatively prime pair of integers. Let $r_{i,j}, s_{i,j}, t_{i,j}$ be as defined in the slow Euclidean algorithm, and $p_{i,j}, q_{i,j}, \delta_{i,j}$ as defined for the sequence of convergents for the continued fraction expansion $a/b = [a_0;a_1,\ldots,a_k]$. Then the following statements hold for $i=0,\ldots,k$ and $j=1,\ldots,a_i$. 
\begin{itemize}
\item We have $a/b \geq p_{i,j}/q_{i,j}$ when $i$ is odd, $a/b \leq p_{i,j}/q_{i,j}$ when $i$ is even.
\item With
\[
 \pm ( a/b - p_{i,j}/q_{i,j} ) < 0,
\]
we have
\eq\label{eq:st_pq}
r_{i,j} = bq_{i,j} \delta_{i,j}, \quad s_{i,j} = \mp p_{i,j}, \quad t_{i,j} = \pm q_{i,j}.
\eeq
\item Moreover, the $a_i$ are the sizes of the blocks in the slow Euclidean algorithm, i.e., $r_i = r_{i-2}-a_ir_{i-1}$.
\end{itemize}
\end{obs}

Another helpful way to express these relationships is to rethink \eqref{eq:EuclidanAlg} for the slow remainders as
\eq\label{eq:r_pq}
r_{i,j} = bt_{i,j}( a/b + s_{i,j}/t_{i,j}) = b|t_{i,j}|\left|a/b - |s_{i,j}/t_{i,j}|\right| = b|t_{i,j}|\delta_{i,j},
\eeq
keeping in mind that $s_{i,j}$ and $t_{i,j}$ have opposite sign.

\begin{table}
\[
 \begin{array}{ |c | c || c| c | c | c || c | c |}
 \hline
  i & j & a_{i} & r_{i,j} & s_{i,j} & t_{i,j} & \uv & \vv\\
  \hline
  \hline
   -1 & 1 &  & 71 & 1 & 0 &  & \\
   \hline
   0 & 1 &  & \mathbf{51} & 0 & \mathbf{1} & \uv_1 = (1,51) & \\
   \hline
   1 & 1 & 1 & \mathbf{20} & 1 & \mathbf{-1} & & \vv_1 = (1,-20) \\
   \hline
   2 & 1 &  & \mathbf{31} & -1 & \mathbf{2} & \uv_2 = (2,31) & \\
    & 2 & 2 & \mathbf{11} & -2 & \mathbf{3} & \uv_3 = (3,11) & \\
   \hline
   3 & 1 & 1 & \mathbf{9} & 3 & \mathbf{-4} & & \vv_2 = (4,-9) \\
   \hline
   4 & 1 & 1 & \mathbf{2} & -5 & \mathbf{7} & \uv_4 = (7,2) & \\
   \hline
   5 & 1 & & \mathbf{7} & 8 & \mathbf{-11} & & \vv_3=(11,-7) \\
    & 2 &  & \mathbf{5} & 13 & \mathbf{-18} & & \vv_4 = (18,-5)\\
    & 3 & & \mathbf{3} & 18 & \mathbf{-25} & & \vv_5 = (25,-3) \\
    & 4 & 4 & \mathbf{1} & 23 & \mathbf{-32} & & \vv_6 = (32,-1)\\
    \hline
    6 & 1 &  & \mathbf{1} & -28 & \mathbf{39} & \uv_5 =(39,1) & \\
     & 2 & 2 & 0 & -51 & 71 & & \\
     \hline
 \end{array}
\]
\caption{The steps of the slow Euclidean algorithm give unit lattice vectors.}\label{tab:eea}
\end{table}

\subsection{Unit lattice vectors}

Let $\uv_1, \uv_2, \ldots, \uv_d$ denote the vectors $\la t_{i,j}, r_{i,j}\ra$, such that $t_{i,j}>0$ are listed in order of appearance in the slow Euclidean algorithm. Similarly, denote the vectors $-\la t_{k,l}, r_{k,l} \ra$, with $t_{k,l} < 0$, by $\vv_1, \vv_2, \ldots, \vv_e$. The set $U=\{ \uv_1, \uv_2, \ldots, \uv_d\}$ is the set of positively sloped pairs of this type, while, $V=\{ \vv_1, \vv_2, \ldots, \vv_e\}$ is the collection of such negatively sloped vectors. We call the vectors in these sets \emph{unit lattice vectors}, since we will prove they are precisely those vectors with lattice length 1.

By construction (since generally speaking the $r_{i,j}$ decrease while the $t_{i,j}$ increase) we have their slopes in decreasing order of magnitude:
\[
 m_{\uv_1} > m_{\uv_2} > \cdots > m_{\uv_d} > 0 > m_{\vv_e} > \cdots > m_{\vv_2} > m_{\vv_1}.
\]
So long as $a < b$, the two steepest vectors are $\uv_1 = \la 1, a\ra$ and $\vv_1 = \la -1, b-a\ra$. Also, notice that each consecutive difference is $\uv_{i+1} - \uv_i =- \la t_{i'}, r_{i'} \ra$ for some $i'$. That is, each difference is a simple $(t,r)$-pair, appearing at the bottom of the block of the algorithm just above $\uv_{i+1}$. This fact can be restated as $\uv_{i+1} - \uv_i = \vv_j$ for some $j$. Likewise, the differences between consecutive entries in $V$ are simple $(t,r)$-pairs from set $U$: $\vv_{k+1}-\vv_k =  \la t_{k'}, r_{k'} \ra = \uv_l$.

In Table \ref{tab:eea} we have the steps of the slow Euclidean algorithm for $a=51$ and $b=71$, with the vectors in sets $U$ and $V$ identified.

\subsection{Basis pairs from the Euclidean algorithm}

It is well known that vectors $\xv=\la x_1, x_2 \ra$ and $\yv = \la y_1, y_2 \ra$ form an integer basis for $L_{a,b}$ if and only if the area of the parallelogram spanned by $\xv$ and $\yv$ is $b$, i.e., if:
\[
 \det\left( \begin{array}{cc} x_1 & y_1 \\ x_2 & y_2 \end{array} \right) = x_1y_2 - y_1x_2 = \pm b.
\]
It so happens that many pairs  of vectors $\langle t_{i,j}, r_{i,j}\rangle$ in the extended Euclidean algorithm have this property.

We begin with a helpful observation, which allows us to focus on the coefficients $s_{i,j}$ and $t_{i,j}$.

\begin{obs}\label{obs:reduce}
 If $r = s b + t a$ and $r' = s'b +t'a$, then 
 \[
 \det\left( \begin{array}{cc} t & t' \\ r & r' \end{array}\right) = (ts'-t's)b,
 \]
 and thus $\{ \la t,r \ra, \la t',r' \ra\}$ is a basis for $L_{a,b}$ if and only if 
 \[
 \det \left(\begin{array}{cc} t & t' \\ s & s'\end{array}\right) = \pm 1.
 \]
\end{obs}

We now provide a collection of bases for $L_{a,b}$ that come from the Euclidean algorithm.

\begin{prop}[Nice basis pairs]\label{prp:basepairs}
Fix integers $a$ and $b$ with $1\leq a < b$. Then for each pair of indices $i$ and $j$ appearing in the slow Euclidean algorithm for $a$ and $b$, we have bases of the form
\[
 \{ \la t_i, r_i \ra, \la t_{i+1,j}, r_{i+1,j} \ra \}, \{ \la t_i, r_i \ra, \la t_{i+2,1}, r_{i+2,1} \ra \}, \mbox{ and } \{ \la t_{i,j}, r_{i,j} \ra, \la t_{i,j+1}, r_{i,j+1} \ra \}.
\]
In other words, the vector at the bottom of one block of the Euclidean algorithm forms a basis with each of the vectors in the next block, as well as the first vector of the block after that. Moreover, adjacent vectors within each block give bases. 

As especially nice cases, consecutive elements of $U$ and $V$ give bases. That is, $\{\uv_i, \uv_{i+1}\}$ is a basis for each $i = 1, \ldots, d-1$, and $\{\vv_j, \vv_{j+1}\}$ is a basis for each $j = 1, \ldots, e-1$.
\end{prop}

\begin{proof}
By Observation \ref{obs:reduce}, it will suffice to show the relevant determinants of $(s,t)$-pairs have absolute value 1.

We make the elementary observation that for any two pairs $(s,t), (s',t')$ such that  
\[
 \det\left(\begin{array}{cc} t & t' \\ s& s'\end{array} \right) = ts'-t's=\pm 1,
\]
then
\[
 \det\left(\begin{array}{cc} t & t' + jt\\ s& s'+js\end{array} \right) = ts'+jts- t's-jts = ts'-t's = \pm 1.
\]
In other words, if $\{ \xv, \yv\}$ form a basis, then $\{ \xv, \yv+j\xv\}$ also forms a basis.

The result now follows by induction, with base cases of $\{ \la 0, b \ra, \la 1, a \ra\}$ and $\{ \la 1, a\ra, \la -1, b-a \ra \}$, since

\[
\{ \la t_i, r_i \ra, \la t_{i+1,j}, r_{i+1,j} \ra \} = \{ \la t_i, r_i \ra, \la t_{i-1}, r_{i-1}\ra-j\la t_i, r_i \ra \},
\]

\[
\{ \la t_i, r_i \ra, \la t_{i+2,1}, r_{i+2,1} \ra \} = \{ \la t_i, r_i \ra, \la t_{i+1}, r_{i+1 }\ra - \la t_i, r_i \ra\}, 
\]
and 
\[
 \{ \la t_{i,j}, r_{i,j} \ra, \la t_{i,j+1}, r_{i,j+1} \ra \} = \{ \la t_{i,j}, r_{i,j} \ra, \la t_{i,j}, r_{i,j} \ra - \la t_{i-1}, r_{i-1} \ra \}.
\]
\end{proof}

Our next goal is to prove that the $(t,r)$-pairs appearing in the Euclidean algorithm give not only nice basis pairs for $L$, but that they are also the vectors of lattice length one.

\subsection{Characterizing unit lattice vectors}

We will now show that vectors in $U \cup V$, i.e., those $\la t_{i,j}, r_{i,j} \ra$ that come from the slow Euclidean algorithm, are the vectors of lattice length one in $L_{a,b}$.

\begin{prop}[Unit lattice length vectors]\label{prop:units}
Let $\xv$ be an increasing lattice vector and let $\yv$ be a decreasing lattice vector. Then $\xv \in U$ if and only if $\ell^+(\xv) = 1$. Likewise, $\yv \in V$ if and only if $\ell^-(\yv) =1$.
\end{prop}

\begin{proof}
We will prove only the statement for the increasing vectors, since the argument is identical in the decreasing case.

To begin we prove that each element $\uv \in U$ has $\ell^+(\uv) = 1$. By Observation \ref{obs:suff1}, it is enough to show that $\shad(\uv_i) = \emptyset$ for all $i$.
Since $\uv_1 = \la 1, a \ra$, it is trivially true that $\shad(\uv_1) = \emptyset$.  Now we will show that if $\uv_i$ has an empty shadow, then $\uv_{i+1}$ has an empty shadow as well. 

Suppose for a contradiction that $\shad(\uv_i) = \emptyset$ but there is a point $(x,y) \in \shad(\uv_{i+1})$. Then either the vector $\xv = \la x, y\ra$ or the vector $\uv_{i+1}-\xv =\la t-x, r-y \ra$ is a vector with slope at least $m_{\uv_{i+1}}$. Without loss of generality, suppose $m_{\xv} \geq m_{\uv_{i+1}}$.  But if this is the case, then the lattice point $(x,y)$ lies in the parallelogram spanned by $\uv_i$ and $\uv_{i+1}$, contradicting the fact that, by Proposition \ref{prp:basepairs}, $\{ \uv_i, \uv_{i+1}\}$ is a basis. 

Hence, we must conclude $\shad(\uv_{i+1}) = \emptyset$, as desired. This proves the first implication.

We now argue for the converse, that if $\xv \notin U$, then $\ell^+(\uv) > 1$. Suppose $\xv$ is an increasing lattice vector. Then for some $i$, $m_{\uv_i} \geq m_{\xv} \geq m_{\uv_{i+1}}$. Since $\{ \uv_i, \uv_{i+1}\}$ is an integer basis, we know $\xv = r\uv_i + s\uv_{i+1}$ for integers $r$ and $s$ where at least one of $r$ or $s$ must be positive since $\xv$ is an increasing lattice vector.

If  $s<0 <r$, then $m_{\xv} > m_{\uv_i}$, a contradiction. Similarly, if $r<0<s$, then $m_{\uv_{i+1}} > m_{\xv}$. Thus both $r$ and $s$ are nonnegative, and $\ell^+(\xv)\geq r+s\geq 1$. If $r=0$ and $s=1$ or vice-versa, then $\xv \in \{ \uv_i, \uv_{i+1}\} \subset U$. Otherwise, $\ell^+(\xv) > 1$, as desired.  
\end{proof}

\subsection{Length in terms of unit vectors}

Having established that the vectors $U\cup V$ are precisely those vectors in $L_{a,b}$ with lattice length one, we now describe the length of any lattice vector in terms of these. To facilitate discussion, throughout the remainder of the section we will focus on increasing vectors only. Similar ideas will yield analogous results for decreasing vectors. 

Let $P=\bigcup_{i=1}^d P_i$ denote the union of the triangles $P_i =\{ \lam \uv_i + \mu \uv_{i+1} : \lam, \mu \in [0,1] \}$. (Note $\lambda, \mu$ are real numbers in this construction.) The boundary of $P$ is the union of segments in the first quadrant given by 
\[
\partial P = \{ \lam \uv_i + (1-\lam)\uv_{i+1} : 0\leq i \leq d, \lam \in [0,1] \},
\]
where $\uv_0=\uv_{d+1} = \mathbf{0}$. Apart from the zero vector, every lattice vector in $P$ has lattice length 1 by Proposition \ref{prop:units}.  Let $kP$ denote the dilation of $P$ by $k$ units. Then
\[
 \partial (kP) = \{ \lam \uv_i + (k-\lam)\uv_{i+1} : 0\leq i \leq d, \lam \in [0,k] \}.
\]

In the proof of Proposition \ref{prop:units}, we showed that if $\xv$ is an increasing vector, then there is an integer $i$ such that $m_{\uv_i} \geq m_{\xv} \geq m_{\uv_{i+1}}$, and moreover we can write $\xv = r\uv_i + s\uv_{i+1}$ for nonnegative integers $r$ and $s$. In other words, this means $\xv \in \partial(kP_i)$, where $k=r+s$.

\begin{obs}\label{obs:Pi}
Every increasing vector $\xv$ lies in $\partial(kP_i)$ for some $i$ and $k$.
\end{obs}

We now show that the dilation of $P$ gives us a way to understand lattice length.

\begin{prop}\label{prp:lengthk}
For any integer $k$, $\ell^+(k\uv_1) = \ell^+(k\uv_d) = k$. If $\xv \in \partial(kP)\cap L_{a,b}$, and if $\xv$ is not a multiple of $\uv_1$ or $\uv_d$, then $\ell^+(\xv)=k$. In particular, if $\xv \in \partial(kP_i)$, then there exist nonnegative integers $c$ and $d$ such that $\xv= c\uv_i + d\uv_{i+1}$ and
 \[
  \ell^+(\xv) = c+d.
 \]
\end{prop}

See Figure \ref{fig:ablatticelength} for an illustration of this result for $a = 51$ and $b=71$.

\begin{proof} 
Since $\uv_1$ has the greatest slope among the increasing unit vectors, multiples of $\uv_1$ have only one increasing path with unit lattice vectors, namely, $\uv_1, \uv_1, \ldots, \uv_1$. Similarly, there is only one increasing path to any multiple of $\uv_d$ since it has the smallest slope. This proves $\ell^+(k\uv_1) = \ell^+(k\uv_d) = k$ as claimed.
 
 We now turn to the points on $\partial (kP)$ that are not multiples of $\uv_1$ or $\uv_d$. We proceed by induction, with the base case of $k=1$ true by Proposition \ref{prop:units}.
 
 Now suppose the claim is true for some integer $k\geq 1$. Let $\xv \in \partial((k+1)P) \cap L_{a,b}$. By Observation \ref{obs:Pi}, there is an $i$ such that $\xv \in \partial((k+1)P_i) \cap L_{a,b}$. Then for some integer $c$ with $0\leq c\leq k+1$, 
\[
\xv = c\uv_i + (k+1-c)\uv_{i+1}.
\]
This shows that $\ell^+(\xv) \geq k+1$.
 
Now, in general, any path in $\wk^+(\xv)$ can be refined into a path consisting only of unit lattice vectors, and we can see
 \[
  \ell^+(\xv) = 1 + \max\{ \ell^+(\xv-\uv) : \uv \in U \mbox{ and } \xv-\uv \in L_{a,b}\}.
 \] 
Now for any $\uv$ such that $\xv-\uv \in L_{a,b}$, the vector $\xv-\uv$ must lie in the interior of $(k+1)P$ since it is in the shadow of $\xv$. Again by Observation \ref{obs:Pi}, this means $\xv-\uv \in \partial(jP)$ for some $j \leq k$, and hence by induction $\ell^+(\xv-\uv)=j \leq k$. Therefore 
\[
 \ell^+(\xv) \leq 1 + k.
\]

Together with our earlier inequality, we have proved the desired result:
\[
 \ell^+(\xv) = k+1.
\]
\end{proof}

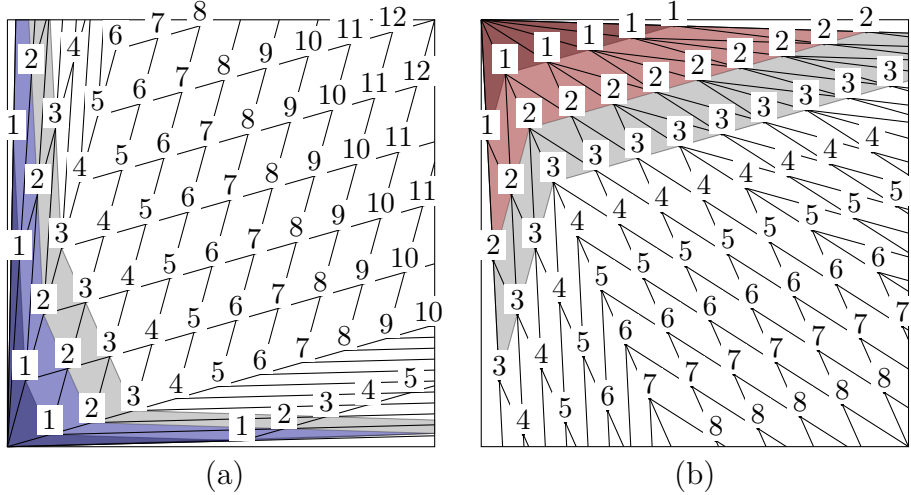
\begin{figure}
\[
\begin{array}{cc}
\begin{tikzpicture}[scale=.08]
\draw[fill=black, opacity=.6] (0,0)--(1,51)--(2,31)--(3,11)--(7,2)--(39,1)--(0,0);
\draw[fill=blue, opacity=.4] (0,0)--(1*71/51,71)--(6 - 1*49/20, 71)--(6,22)--(14,4)--(14+32*57/32, 4-1*57/32)--(39,1)--(0,0);
\draw[fill=gray, opacity=.4] (0,0)--(1*71/51,71)--(9 - 1*38/20, 71)--(9,33)--(21,6)--(21+32*50/32, 6-1*50/32)--(14+32*57/32, 4-1*57/32)--(0,0);
\draw (0,0)-- (71*1/51,71);
\draw (0,0)-- (71*2/31,71);
\draw (3,11)--(3 + 2*60/31, 11+60);
\draw (6,22)--(2*3 + 2*49/31, 2*11+31*49/31);
\draw (3*3,3*11)--(3*3 + 2*38/31, 3*11+31*38/31);
\draw (4*3,4*11)--(4*3 + 2*27/31, 4*11+31*27/31);
\draw (5*3,5*11)--(5*3 + 2*16/31, 5*11+31*16/31);
\draw (6*3,6*11)--(6*3 + 2*5/31, 6*11+31*5/31);
\draw (0,0)--(3*71/11, 71);
\draw (2,31)--(2 + 3*40/11, 31+11*40/11);
\draw (4,62)--(4 + 3*9/11, 62+11*9/11);
\foreach \x in {1,...,8}{
\draw (7*\x,2*\x)--(7*\x + 3*71/11-6*\x/11, 71);
}
\draw (7*9,2*9) --(7*9+3*8/3, 2*9+11*8/3);
\draw (0,0)--(71,2*71/7);
\foreach \x in {1,...,4}{
\draw (3*\x,11*\x)--(71, 11*\x + 2*71/7-6*\x/7 );
}
\draw (3*5, 11*5)--(3*5 + 7*16/2 , 71);
\draw (3*6, 11*6)--(3*6 + 2*7, 11*6+2*2);
\foreach \x in {1,...,10}{
\draw (7*\x,2*\x)--(71, 2*\x + 71/39-7*\x/39 );
}
\draw (39,1)--(71, 1+2*32/7);
\draw (0,0)--(71,71/39);
\draw (0,71)--(0,0)--(71,0)--(71,71)--(0,71);
\draw (1,51) node[above,inner sep=2,fill=white] {\small $1$};
\draw (2,31) node[above,inner sep=2,fill=white] {\small $1$};
\foreach \x in {1,...,6}{
\draw (3*\x,11*\x) node[above,inner sep=2, fill=white] {\small $\x$};
}
\foreach \x in {1,...,10}{
\draw (7*\x,2*\x) node[above,inner sep=2, fill=white] {\small $\x$};
}
\foreach \x in {2,...,10}{
\draw (-4+7*\x,9+2*\x) node[above,inner sep=2, fill=white] {\small $\x$};
}
\foreach \x in {3,...,11}{
\draw (-8+7*\x,18+2*\x) node[above,inner sep=2, fill=white] {\small $\x$};
}
\foreach \x in {4,...,11}{
\draw (-12+7*\x,27+2*\x) node[above,inner sep=2, fill=white] {\small $\x$};
}
\foreach \x in {5,...,12}{
\draw (-16+7*\x,36+2*\x) node[above,inner sep=2, fill=white] {\small $\x$};
}
\foreach \x in {6,...,12}{
\draw (-20+7*\x,45+2*\x) node[above,inner sep=2, fill=white] {\small $\x$};
}
\foreach \x in {7,...,8}{
\draw (-24+7*\x,54+2*\x) node[above,inner sep=2, fill=white] {\small $\x$};
}
\foreach \x in {2,...,5}{
\draw (32+7*\x,-1+2*\x) node[above,inner sep=2, fill=white] {\small $\x$};
}
\draw (39,1) node[above,inner sep=2, fill=white] {\small $1$};
\draw (5,42) node[above,inner sep=2, fill=white] {\small $2$};
\draw (4,62) node[above,inner sep=2, fill=white] {\small $2$};
\draw (8,53) node[above,inner sep=2, fill=white] {\small $3$};
\draw (11,64) node[above,inner sep=2, fill=white] {\small $4$};
\end{tikzpicture}
&
\begin{tikzpicture}[scale=.08]
\draw (0,0)--(71,0)--(71,-71)--(0,-71)--(0,0);
\foreach \x in {1,...,70}{
 \draw[fill=black] (\x, {mod(\x*51,71)-71} ) circle (5pt);
}
\draw[fill=black, opacity=.6] (0,0)--(1,-20)--(4,-9)--(11,-7)--(18,-5)--(25,-3)--(32,-1)--(0,0);
\draw[fill=red, opacity=.4] (0,0)--(2,-40)--(8, -18)--(22,-14)--(36,-10)--(50, -6)--(64,-2)--(0,0);
\draw[fill=gray, opacity=.4] (0,0)--(3,-60)--(12,-27)--(33,-21)--(54,-15)--(71,-15 + 17/7*2)--(71, -1*71/32 )--(0,0);
\draw (0,0)--(71/20*1,-71);
\draw (0,0)--(4*71/9,-71);
\draw (0,0)--(71,-7*71/11);
\draw (0,0)--(71,-5*71/18);
\draw (0,0)--(71,-3*71/25);
\draw (0,0)--(71,-1*71/32);
\draw (4,-9)--(4+62/20,-71);
\draw (8,-18)--(8+53/20 ,-71);
\draw (12,-27)--(12+44/20 ,-71);
\draw (16,-36)--(16+35/20 ,-71);
\draw (20,-45)--(20+26/20 ,-71);
\draw (24,-54)--(24+17/20 ,-71);
\draw (28,-63)--(28+8/20 ,-71);
\draw (11,-7)--(11+4*64/9, -71);
\draw (22,-14)--(22+4*57/9, -71);
\draw (33,-21)--(33+4*50/9, -71);
\draw (44,-28)--(44+4*43/9, -71);
\draw (55,-35)--(55+4*36/9, -71);
\draw (66,-42)--(71, -42-9*5/4);
\draw (18,-5)--(71,-5-7*53/11);
\draw (18,-5)--(71,-5-3*53/25);
\draw (36,-10)--(71,-10-3*35/25);
\draw (36,-10)--(71,-10-7*35/11);
\draw (54,-15)--(71,-15-7*17/11);
\draw (54, -15)--(71,-15-3*17/25);
\draw (25,-3)--(71,-3-5*46/18);
\draw (25,-3)--(71, -3 -46/32);
\draw (50,-6)--(71, -6 -21/32);
\draw (50,-6)--(71,-6-5*21/18);
\draw (32,-1)--(71,-1-3*39/25);
\draw (64,-2)--(71,-2-3*7/25);
\draw (1,-20)--(1+4*51/9, -71);
\draw (2,-40)--(2+4*31/9, -71);
\draw (3,-60)--(3+4*11/9, -71);
\draw (4,-9)--(71, -9-7*67/11);
\draw (8,-18)--(71, -18-7*63/11);
\draw (12,-27)--(71, -27-7*59/11);
\draw (16,-36)--(71, -71);
\draw (20,-45)--(20+11*26/7, -71);
\draw (24,-54)--(24+11*17/7, -71);
\draw (28,-63)--(28+11*8/7, -71);
\draw (11, -7)--(71, -7-5*60/18);
\draw (22,-14)--(71, -14 -5*49/18);
\draw (33,-21)--(71, -21 -5*38/18);
\draw (44,-28)--(71, -28 -5*27/18);
\draw (55,-35)--(71, -35 -5*16/18);
\draw (66,-42)--(71, -42 -5*5/18);
\draw (1,-20) node [above,inner sep=2, fill=white] {\small $1$};
\draw (4,-9) node [above,inner sep=2, fill=white] {\small $1$};
\draw (11,-7) node [above,inner sep=2, fill=white] {\small $1$};
\draw (18,-5) node [above,inner sep=2, fill=white] {\small $1$};
\draw (25,-3) node [above,inner sep=2, fill=white] {\small $1$};
\draw (32,-1) node [above,inner sep=2, fill=white] {\small $1$};
\foreach \x in {0,1,2}{
 \draw (2+\x*3,31+\x*11-71) node [above,inner sep=2, fill=white] {\small $2$};
} 
\foreach \x in {1,...,8}{
 \draw (2+2*3+\x*7,31+2*11-71+\x*2) node [above,inner sep=2, fill=white] {\small $2$};
} 
\foreach \x in {0,1,2,3}{
 \draw (3+\x*3,11+\x*11-71) node [above,inner sep=2, fill=white] {\small $3$};
} 
\foreach \x in {1,...,8}{
 \draw (12+\x*7,44-71+\x*2) node [above,inner sep=2, fill=white] {\small $3$};
}
\foreach \x in {0,1,2,3}{
 \draw (7+\x*3,2+\x*11-71) node [above,inner sep=2, fill=white] {\small $4$};
} 
\foreach \x in {1,...,7}{
 \draw (7+3*3+\x*7,2+3*11-71+\x*2) node [above,inner sep=2, fill=white] {\small $4$};
}
\foreach \x in {0,1,2}{
 \draw (14+\x*3,4+\x*11-71) node [above,inner sep=2, fill=white] {\small $5$};
} 
\foreach \x in {1,...,7}{
 \draw (14+2*3+\x*7,4+2*11-71+\x*2) node [above,inner sep=2, fill=white] {\small $5$};
} 
\foreach \x in {0,1}{
 \draw (21+\x*3,6+\x*11-71) node [above,inner sep=2, fill=white] {\small $6$};
}
\foreach \x in {1,...,6}{
 \draw (24+\x*7,17-71+\x*2) node [above,inner sep=2, fill=white] {\small $6$};
}
\foreach \x in {0,...,6}{
 \draw (24+\x*7+4,17-71+\x*2-9) node [above,inner sep=2, fill=white] {\small $7$};
}
\foreach \x in {0,...,4}{
 \draw (39+\x*7,1-71+\x*2) node [above,inner sep=2, fill=white] {\small $8$};
}
\end{tikzpicture}
\\
\mbox{(a)} & \mbox{(b)}
\end{array}
\]
\caption{The lattice $L_{51,71}$, with points labeled according to (a) increasing lattice length, and (b) decreasing lattice length. Dilations of the polygon $P$ are shaded in blue in (a), and red in (b).}\label{fig:ablatticelength}
\end{figure}

In later sections, we will want to have a different description of the location of a vector $\xv$ that also determines its length. Let us denote three consecutive unit vectors from simple $(t,r)$-pairs as
\[
\uv = \langle t_{2i-2}, r_{2i-2}\rangle, \vv=\langle t_{2i-1}, -r_{2i-1}\rangle, \mbox{ and } \wv = \langle t_{2i}, r_{2i} \rangle,
\] 
so that $\wv = \uv + a_{2i}\vv$, where $a_{2i}$ is the corresponding continued fraction term. Then every vector of the form $\uv + a\vv$, with $0\leq a\leq a_{2i}$, is a unit lattice vector corresponding to a slow $(t,r)$-pair. If $\uv =\uv_j$, then $\uv + \vv= \uv_{j+1}$, $\uv + 2\vv = \uv_{j+2}$, and so on, with $\wv = \uv_{j+a_{2i}}$. Thus, by Proposition \ref{prp:lengthk}, any lattice vector in the $k$-fold dilation of the line segment $(1-\lambda)\uv + \lambda\wv$, $(0\leq \lambda \leq 1)$ will also have length $k$, since this segment is just a union of dilations of segments between consecutive unit lattice vectors. We capture this idea in the following corollary to Proposition \ref{prp:lengthk}.

\begin{corollary}\label{cor:simpleangle}
Let $\uv$, $\vv$, and $\wv$, as above, with $\wv = \uv + a_{2i}\vv$. Then the following are equivalent for an increasing vector $\xv$:
\begin{itemize}
\item Vector $\xv$ has $\ell^+(\xv) = k$ and slope bounded by that of $\uv$ and $\wv$: $m_{\wv} \leq m_{\xv} \leq m_{\uv}$.
\item There exists an integer $0\leq a \leq ka_{2i}$ such that $\xv = k\uv + a \vv = k\wv-(ka_{2i}-a)\vv$.
\end{itemize}
\end{corollary}

\subsection{Best unit vectors to approximate a given slope}

To use our notions of lattice length to find longest increasing/decreasing subsequences it is important for us to consider the unit vectors with slopes closest to $\pm 1$. In fact, we do better, and find the analogous unit lattice vectors for any slope $\pm \tau$. 

From this point forward, fix a real slope $\tau \geq 1$ (our illustrations will use $\tau =1$ for ease of viewing) and let $\av$, $\bv$, $\cv$, and $\dv$ denote those consecutive unit lattice vectors such that
\[
 m_{\av} \geq  \tau > m_{\bv} \quad \mbox{ and } \quad m_{\cv} > -\tau \geq m_{\dv}.
\]
Note $\av = \uv_i$ and $\bv=\uv_{i+1}$ for some $i$, while $\cv=\vv_{j+1}$ and $\dv=\vv_{j}$ for some $j$. We will show that $\cv$ and $\dv$ are nice linear combinations of $\av$ and $\bv$.

For example, in Table \ref{tab:eea}, we see that for $a=51$ and $b=71$ and $\tau=1$, we have $\av = \uv_3 = \la 3,11 \ra$,  $\bv = \uv_4 = \la 7,2 \ra$, $\cv=\vv_3= \la 11,-7 \ra$, and $\dv = \vv_2 = \la 4,-9 \ra$.

First we make some observations about the slow Euclidean algorithm in order to describe where $\av, \bv, \cv$, and $\dv$ occur.  The idea behind the following observations is merely that ``simple remainders decrease'' and ``remainders decrease within subtraction blocks'' while ``$t$-coefficients increase in absolute value and alternate signs from block to block.''

\begin{obs}\label{obs:abcd}
We have the following characterizations. 
\begin{itemize}
\item Let $i$ and $j$ be the lexicographically first pair such that  $\tau > r_{i,j}/t_{i,j}>0$. Then $\bv = \la t_{i,j}, r_{i,j}\ra$ is the unit vector with greatest slope less than $\tau$ and 
\[
\av = \begin{cases}
   \la t_{i,j-1}, r_{i,j-1} \ra & \mbox{ if  $j>1$,}\\
   \la t_{i-2}, r_{i-2} \ra & \mbox{ if $j=1$.}
 \end{cases}
\]
\item Let $k$ and $l$ be the lexicographically first pair such that $\tau > -r_{k,l}/t_{k,j} > 0$. Then $\cv = \la -t_{k,l}, -r_{k,l}\ra $ is the unit vector with least slope greater than $-\tau$ and
\[
 \dv = \begin{cases}
   \la -t_{k,l-1}, -r_{k,l-1}) & \mbox{ if  $l>1$,}\\
   \la -t_{k-2}, -r_{k-2} \ra & \mbox{ if $l=1$.}
 \end{cases}
\]
\end{itemize}
\end{obs}

From now on, let us write  $\xv=\cv-\dv$ and $\yv = \bv - \av$. As discussed Section \ref{sec:slowEuclid}, it follows from the definition of the slow Euclidean algorithm that $\xv \in U$ and $\yv\in V$ are unit lattice vectors and moreover that they are simple remainder pairs (of the form $\la t_i, r_i \ra$). These vectors will provide a convenient way for us to frame our results. 

\begin{obs}\label{obs:bc}
 Let $i$ be minimal such that $\tau > r_i/|t_i| = b\delta_i >0$. 
\begin{enumerate}
\item If $t_i > 0$, i.e., if $i$ is even, then $\bv$ is in the $i$th block: $\bv = \la t_{i,j}, r_{i,j} \ra$ for some $j$, and $\cv$ is in the $i+1$st block: $\cv = -\la t_{i+1,k},r_{i+1,k} \ra$ for some $k$. Moreover:
\begin{itemize}
\item $\yv = -\la t_{i-1}, r_{i-1}\ra $, $t_{i-1} < 0$, and by minimality of $i$, $\tau \leq r_{i-1}/|t_{i-1}|=-m_{\yv}$,
\item $\xv = \la t_i, r_i \ra$, with $m_{\xv} \leq m_{\bv} < \tau$.
\end{itemize}
\item If $t_i < 0$, i.e., if $i$ is odd, then $\cv$ is in the $i$th block: $\cv = -\la t_{i,j},r_{i,j}\ra$ for some $j$, and $\bv$ is in the $i+1$st block: $\bv = \la t_{i+1,k}, r_{i+1,k} \ra$ for some $k$. Moreover:
\begin{itemize}
\item $\yv =- \la t_i, r_i\ra$, with $0 <  -m_{\yv}=r_i/|t_i| <\tau$,
\item $\xv = \la t_{i-1}, r_{i-1} \ra$, $t_{i-1}>0$, and by minimality of $i$, $\tau \leq r_{i-1}/t_{i-1} = m_{\xv}$.
\end{itemize}
\end{enumerate}
In either case, precisely one of $\xv$ and $\yv$ has slope of absolute value less than $\tau$. Moreover, as $\xv$ and $\yv$ come from simple $(t,r)$-pairs in adjacent blocks, $\{\xv, \yv\}$ is a basis for $L_{a,b}$ by Proposition \ref{prp:basepairs}.
\end{obs}

Now that we have very tightly identified where vectors $\av, \bv, \cv, \dv, \xv$, and $\yv$ appear in the slow Euclidean algorithm, we are ready to characterize the linear relationships between them. There are four cases, which are equivalent up to transformations of the lattice.

\begin{prop}\label{prp:shortbases}
We have the following expressions relating $\av, \bv, \cv$, and $\dv$.
\begin{enumerate}
 \item If $m_{\yv} \leq -\tau$, then $0 < m_{\xv} < \tau$. Furthermore, 
 \begin{enumerate}
  \item if $|\yv|\leq |\xv|$, then there exists an integer $s\geq 1$ such that:
   \[
    \left(\begin{array}{cc} \av & \cv \\ \bv & \dv \end{array}\right) = \left(\begin{array}{cc} \xv - s\yv & \yv + \xv \\ \xv - (s-1)\yv & \yv \end{array}\right) = \left(\begin{array}{cc} \cv-(s+1)\dv & (s+1)\bv-s\av\\ \cv-s\dv & \bv-\av \end{array}\right).
   \] 
  \item if $|\yv| > |\xv|$, then there exists an integer $s\geq 1$ such that:
   \[
    \left(\begin{array}{cc} \av & \cv \\ \bv & \dv \end{array}\right) = \left(\begin{array}{cc} -\yv + \xv & \yv+s\xv \\ \xv & \yv+(s-1)\xv \end{array}\right) = \left(\begin{array}{cc} s\cv-(s+1)\dv & (s+1)\bv - \av \\ \cv - \dv & s\bv-\av \end{array}\right).
   \] 
   \end{enumerate}
 \item If $0> m_{\yv} >  -\tau$, then $m_{\xv} \geq \tau$. Furthermore,
 \begin{enumerate}
  \item if $|\yv|\leq |\xv|$, then there exists an integer $s\geq 1$ such that:
   \[
    \left(\begin{array}{cc} \av & \cv \\ \bv & \dv \end{array}\right) = \left(\begin{array}{cc} \xv + (s-1)\yv & \yv \\ \xv +s\yv & \yv-\xv \end{array}\right) = \left(\begin{array}{cc} s\cv-\dv & \bv - \av \\ (s+1)\cv - \dv  & s\bv-(s+1)\av \end{array}\right).
   \] 
  \item if $|\yv| > |\xv|$, then there exists an integer $s\geq 1$ such that:
   \[
    \left(\begin{array}{cc} \av & \cv \\ \bv & \dv \end{array}\right) = \left(\begin{array}{cc} \xv & \yv-(s-1)\xv \\ \yv+\xv & \yv-s\xv \end{array}\right) = \left(\begin{array}{cc} \cv-\dv & \bv - s\av \\ (s+1)\cv - s\dv  & \bv-(s+1)\av \end{array}\right).
   \] 
 \end{enumerate}
\end{enumerate}
\end{prop}

\begin{proof}
 We will prove case (1a) in detail. The other cases follow from symmetries, as indicated in Figure \ref{fig:cases}.
 
 Suppose that $-\yv = \av - \bv = \la t_{i-1}, r_{i-1}\ra $ with $t_{i-1} < 0$ and $\tau \leq  r_{i-1}/|t_{i-1}| = -m_{\yv}$, whereas $\xv = \cv - \dv = \la t_i, r_i\ra$ with $\tau > r_i/t_i = m_{\xv}$.  
From Observation \ref{obs:bc}, $\xv$ and $\yv$ are of unit lattice length and form an integer basis for the lattice. Moreover, by considering their slopes, we know that $\dv$ and $\cv$ are in the nonnegative span of $\xv$ and $\yv$. Further, we suppose that $|\xv| \geq |\yv|$. Since $\xv$ is positively sloped and $\yv$ is negatively sloped with $|\yv|\leq |\xv|$, we have 
 \[
 m_{\xv+\yv} > -1 \geq -\tau.
 \]
 
 We claim that $\yv = \dv$. Indeed, suppose $\dv'$ is a lattice vector whose slope satisfies $m_{\dv} < m_{\dv'} \leq -\tau$. Then $\dv' = k\yv + l \xv = (k-l)\yv + l(\xv + \yv)$ for some nonnegative integers $k$ and $l$. Since $m_{\xv+\yv} > -\tau$ and $\yv$ and $\dv'$ have slopes less than or equal to $-\tau$, we see that $k > l > 0$. But also the vector $(k-1)\yv + l\xv$ must have negative slope (else, it is positively sloped and $|\yv| > |(k-1)\yv + l\xv| > |\xv|$), so we can write $\dv' = ((k-1)\yv + l\xv) + \yv$, the sum of two negatively sloped lattice vectors. This shows the lattice length of any such $\dv'$ has $\ell^-(\dv') > 1$.

Now, given that $\dv = \yv$, we have $\cv = \yv + \xv$ by the definition of $\xv$.

Observation \ref{obs:bc} tells us that in the context of the slow Euclidean algorithm, $\xv = \la t_i, r_i \ra$ is the simple pair at the bottom of the block containing $\bv$. Thus adding some number of copies of $-\yv$ to $\xv$ will lead to $\bv$ (say $s-1$ of them) and by definition, one more copy gives $\av$. This proves
\[
 \bv = \xv - (s-1)\yv \quad \mbox{ and } \quad \av = \xv - s\yv.
\]
The expressions for $\av$ and $\bv$ in terms of $\cv$ and $\dv$ (and vice-versa) now follow from easy algebra.
\end{proof}

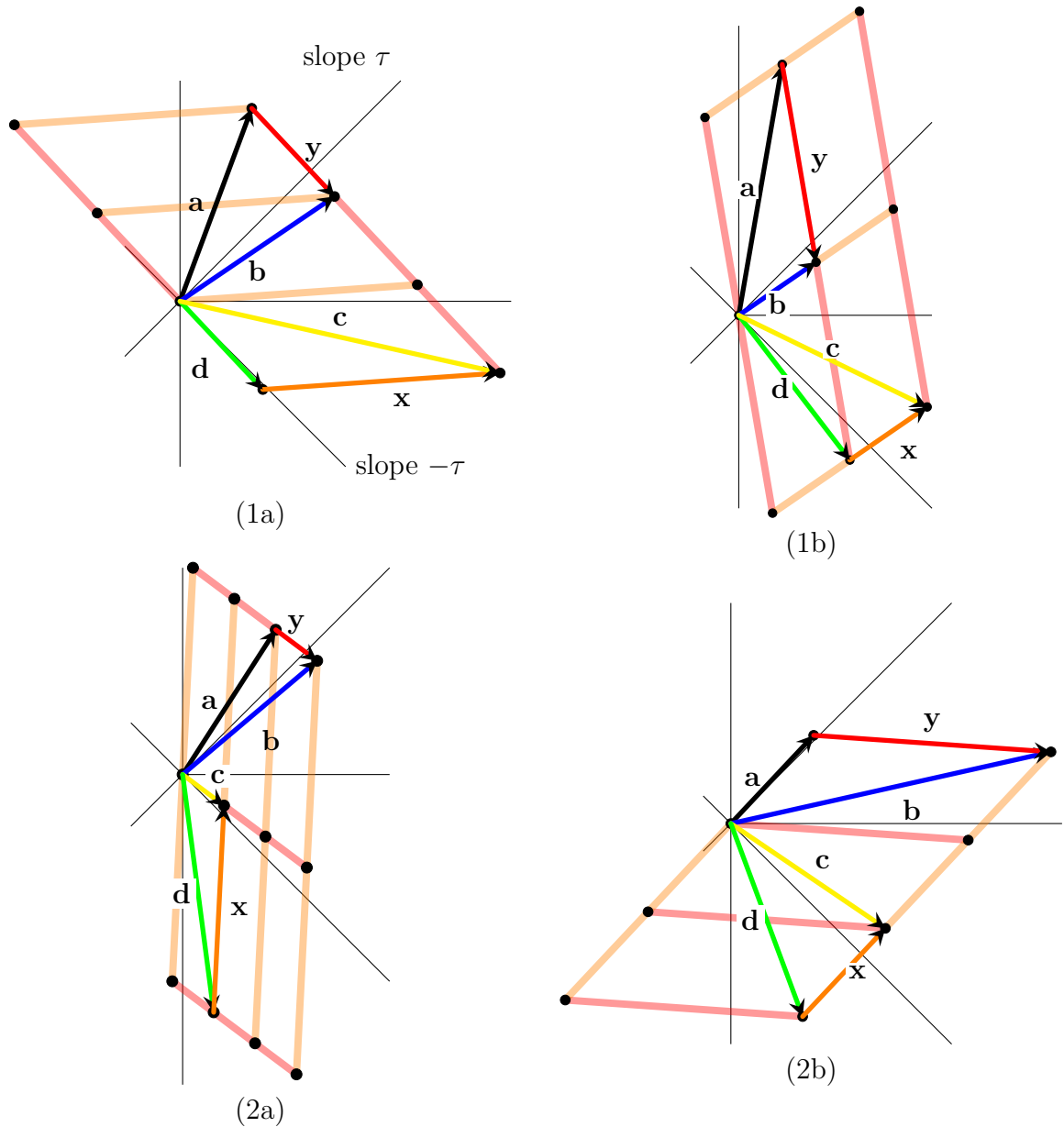
\begin{figure}
\[
\begin{array}{cc}
 \begin{array}{c}
  \begin{tikzpicture}[scale=.08,>=stealth]
\draw (0,40)--(0,0)--(60,0);
\draw (-10,-10)--(40,40) node[above left] {slope $\tau$};
\draw (-10,10)--(30,-30) node[right] {slope $-\tau$};
\draw (0,0)--(0,-30);
\draw[line width=3, cap =round, draw=red, opacity=.4] (-15,16)--(0,0);
\draw[line width=3, cap =round, draw=red, opacity=.4] (28,19)--(43,3);
\draw[line width=3, cap =round, draw=red, opacity=.4] (43,3)--(58,-13);
\draw[line width=3, cap =round, draw=red, opacity=.4] (-15,16)--(-30,32);
\draw[line width=3, cap =round, draw=orange, opacity=.4] (-30,32)--(13,35);
\draw[line width=3, cap =round, draw=orange, opacity=.4] (-15,16)--(28,19);
\draw[line width=3, cap =round, draw=orange, opacity=.4] (0,0)--(43,3);
\draw[fill=black] (13,35) circle (25pt);
\draw[fill=black] (28,19) circle (25pt);
\draw[fill=black] (58,-13) circle (25pt);
\draw[fill=black] (15,-16) circle (25pt);
\draw[fill=black] (-15,16) circle (25pt);
\draw[fill=black] (-30,32) circle (25pt);
\draw[fill=black] (43,3) circle (25pt);
\draw[fill=black] (0,0) circle (25pt);
\draw[line width=2, cap =round, draw=black, ->] (0,0)--(13,35) node[midway,left] {$\av$};
\draw[line width=2, cap =round, draw=blue, ->] (0,0)--(28,19) node[midway,below] {$\bv$};
\draw[line width=2, cap =round, draw=orange, ->] (15,-16)--(58,-13) node[midway,below right] {$\xv$};
\draw[line width=2, cap =round, draw=red, ->] (13,35)--(28,19) node[midway,right] {$\yv$};
\draw[line width=2, cap =round, draw=green, ->] (0,0)--(15,-16) node[midway, below left] {$\dv$};
\draw[line width=2, cap =round, draw=yellow, ->] (0,0)--(58,-13) node[midway,above] {$\cv$};
\end{tikzpicture}
\\
\mbox{(1a)}
 \end{array}
 &
  \begin{array}{c}
\begin{tikzpicture}[scale=.07,>=stealth]
\draw (0,60)--(0,0)--(40,0);
\draw (-10,-10)--(40,40);
\draw (-10,10)--(40,-40);
\draw (0,0)--(0,-40);
\draw[line width=3, cap =round, draw=red, opacity=.4] (7,-41)--(0,0);
\draw[line width=3, cap =round, draw=red, opacity=.4] (-7,41)--(0,0);
\draw[line width=3, cap =round, draw=red, opacity=.4] (16,11)--(23,-30);
\draw[line width=3, cap =round, draw=red, opacity=.4] (32,22)--(39,-19);
\draw[line width=3, cap =round, draw=red, opacity=.4] (32,22)--(25,63);
\draw[line width=3, cap =round, draw=orange, opacity=.4] (-7,41)--(9,52);
\draw[line width=3, cap =round, draw=orange, opacity=.4] (25,63)--(9,52);
\draw[line width=3, cap =round, draw=orange, opacity=.4] (7,-41)--(23,-30);
\draw[line width=3, cap =round, draw=orange, opacity=.4] (16,11)--(32,22);
\draw[fill=black] (9,52) circle (25pt);
\draw[fill=black] (16,11) circle (25pt);
\draw[fill=black] (39,-19) circle (25pt);
\draw[fill=black] (23,-30) circle (25pt);
\draw[fill=black] (-7,41) circle (25pt);
\draw[fill=black] (7,-41) circle (25pt);
\draw[fill=black] (32,22) circle (25pt);
\draw[fill=black] (25,63) circle (25pt);
\draw[fill=black] (0,0) circle (25pt);
\draw[line width=2, cap =round, draw=black, ->] (0,0)--(9,52) node[midway,left,fill=white, inner sep =1] {$\av$};
\draw[line width=2, cap =round, draw=blue, ->] (0,0)--(16,11) node[midway,below, fill=white, inner sep=1] {$\bv$};
\draw[line width=2, cap =round, draw=orange, ->] (7+16,-41+11)--(7+32,-41+22) node[midway,below right] {$\xv$};
\draw[line width=2, cap =round, draw=red, ->] (9,52)--(16,11) node[midway,right] {$\yv$};
\draw[line width=2, cap =round, draw=green, ->] (0,0)--(7+16,11-41) node[midway, left, fill=white, inner sep=1] {$\dv$};
\draw[line width=2, cap =round, draw=yellow, ->] (0,0)--(7+32,22-41) node[midway,above, fill=white, inner sep =1] {$\cv$};
\end{tikzpicture}
\\
\mbox{(1b)}
 \end{array}\\
  \begin{array}{c}
\begin{tikzpicture}[scale=.15,>=stealth]
\draw (0,20)--(0,0)--(20,0);
\draw (-5,-5)--(20,20);
\draw (-5,5)--(20,-20);
\draw (0,0)--(0,-30);
\draw[line width=3, cap =round, draw=orange, opacity=.4] (0,0)--(1,20);
\draw[line width=3, cap =round, draw=orange, opacity=.4] (0,0)--(-1,-20);
\draw[line width=3, cap =round, draw=orange, opacity=.4] (4,-3)--(5,17);
\draw[line width=3, cap =round, draw=orange, opacity=.4] (8,-6)--(9,14);
\draw[line width=3, cap =round, draw=orange, opacity=.4] (12,-9)--(13,11);
\draw[line width=3, cap =round, draw=orange, opacity=.4] (8,-6)--(7,-26);
\draw[line width=3, cap =round, draw=orange, opacity=.4] (12,-9)--(11,-29);
\draw[line width=3, cap =round, draw=red, opacity=.4] (1,20)--(5,17);
\draw[line width=3, cap =round, draw=red, opacity=.4] (5,17)--(9,14);
\draw[line width=3, cap =round, draw=red, opacity=.4] (4,-3)--(8,-6);
\draw[line width=3, cap =round, draw=red, opacity=.4] (8,-6)--(12,-9);
\draw[line width=3, cap =round, draw=red, opacity=.4] (-1,-20)--(3,-23);
\draw[line width=3, cap =round, draw=red, opacity=.4] (3,-23)--(7,-26);
\draw[line width=3, cap =round, draw=red, opacity=.4] (7,-26)--(11,-29);
\draw[fill=black] (9,14) circle (15pt);
\draw[fill=black] (13,11) circle (15pt);
\draw[fill=black] (4,-3) circle (15pt);
\draw[fill=black] (3,-23) circle (15pt);
\draw[fill=black] (1,20) circle (15pt);
\draw[fill=black] (5,17) circle (15pt);
\draw[fill=black] (8,-6) circle (15pt);
\draw[fill=black] (12,-9) circle (15pt);
\draw[fill=black] (-1,-20) circle (15pt);
\draw[fill=black] (7,-26) circle (15pt);
\draw[fill=black] (11,-29) circle (15pt);
\draw[fill=black] (0,0) circle (15pt);
\draw[line width=2, cap =round, draw=black, ->] (0,0)--(9,14) node[midway, left] {$\av$};
\draw[line width=2, cap =round, draw=blue, ->] (0,0)--(13,11) node[midway, below right] {$\bv$};
\draw[line width=2, cap =round, draw=red, ->] (9,14)--(13,11) node[midway, above] {$\yv$};
\draw[line width=2, cap =round, draw=yellow, ->] (0,0)--(4,-3) node[midway, above right, fill=white, inner sep=2] {$\cv$};
\draw[line width=2, cap =round, draw=green, ->] (0,0)--(3,-23) node[midway, left, fill=white, inner sep=2] {$\dv$};
\draw[line width=2, cap =round, draw=orange, ->] (3,-23)--(4,-3) node[midway, right] {$\xv$};
\end{tikzpicture}
\\
\mbox{(2a)}
 \end{array}
 &
  \begin{array}{c}
\begin{tikzpicture}[scale=.08,>=stealth]
\draw (0,40)--(0,0)--(60,0);
\draw (-5,-5)--(40,40);
\draw (-5,5)--(40,-40);
\draw (0,0)--(0,-40);
\draw[line width=3, cap =round, draw=red, opacity=.4] (0,0)--(43,-3);
\draw[line width=3, cap =round, draw=red, opacity=.4] (-15,-16)--(-15+43,-16-3);
\draw[line width=3, cap =round, draw=red, opacity=.4] (-15*2,-16*2)--(-15*2+43,-16*2-3);
\draw[line width=3, cap =round, draw=orange, opacity=.4] (0,0)--(-15,-16);
\draw[line width=3, cap =round, draw=orange, opacity=.4] (-15,-16)--(-30,-32);
\draw[line width=3, cap =round, draw=orange, opacity=.4] (28,-19)--(43,-3);
\draw[line width=3, cap =round, draw=orange, opacity=.4] (58,13)--(43,-3);
\draw[fill=black] (15,16) circle (25pt);
\draw[fill=black] (58,13) circle (25pt);
\draw[fill=black] (28,-19) circle (25pt);
\draw[fill=black] (13,-35) circle (25pt);
\draw[fill=black] (-15,-16) circle (25pt);
\draw[fill=black] (-30,-32) circle (25pt);
\draw[fill=black] (43,-3) circle (25pt);
\draw[fill=black] (0,0) circle (25pt);
\draw[line width=2, cap =round, draw=black, ->] (0,0)--(15,16) node[midway, left] {$\av$};
\draw[line width=2, cap =round, draw=blue, ->] (0,0)--(58,13) node[midway, below right] {$\bv$};
\draw[line width=2, cap =round, draw=red, ->] (15,16)--(58,13) node[midway, above] {$\yv$};
\draw[line width=2, cap =round, draw=yellow, ->] (0,0)--(28,-19) node[midway, above right, fill=white, inner sep=2] {$\cv$};
\draw[line width=2, cap =round, draw=green, ->] (0,0)--(13,-35) node[midway, left, fill=white, inner sep=2] {$\dv$};
\draw[line width=2, cap =round, draw=orange, ->] (13,-35)--(28,-19) node[midway, right,fill=white, inner sep=1] {$\xv$};
\end{tikzpicture}
\\
\mbox{(2b)}
 \end{array}
\end{array}
\]
\caption{The four cases in Proposition \ref{prp:shortbases}.}\label{fig:cases}
\end{figure}

\subsection {The vectors $\xv$ and $\yv$ in terms of the convergents for $a/b$.}\label{xy_conv}

The vectors $\xv$ and $\yv$ will play an important role in what follows, and they admit a nice interpretation in terms of the continued fraction expansion of $a/b$. From Observation \ref{obs:SlowCF}, and specifically \eqref{eq:st_pq} and \eqref{eq:r_pq}, we have generally that 
\[
 \langle t_{i,j}, r_{i,j} \rangle = \langle q_{i,j}, \pm bq_{i,j} \delta_{i,j} \rangle =  \langle q_{i,j}, \pm bq_{i,j}|a/b - p_{i,j}/q_{i,j}| \rangle.
\]

Since $\yv$ is a decreasing vector whose entries are a simple $(t,r)$-pair, we have $\yv = -\langle t_i, r_i\rangle$, with $t_i = q_{2h+1}$ for some $h$, i.e.,
\eq\label{eq:xv_q}
\yv =\langle q_{2h+1}, -bq_{2h+1}\delta_{2h+1}\rangle = \left\langle q_{2h+1},-bq_{2h+1}\left|a/b - p_{2h+1}/q_{2h+1}\right| \right\rangle.
\eeq
Since $\xv$ and $\yv$ come from adjacent blocks of the slow Euclidean algorithm, we find the similar formula for $\xv$:
\eq\label{eq:yv_q}
\xv = \langle q_j, bq_j\delta_j\rangle = \left\langle q_{j},bq_j\left|a/b-p_j/q_j\right| \right\rangle,
\eeq
where $j=2h$ or $j=2h+2$. 

Though it has not yet been emphasized, we now notice that the slope of any unit lattice vector $\langle t_{i,j}, r_{i,j}\rangle$ is $m = \pm b\delta_{i,j}$. Thus, following Observation \ref{obs:bc}, we let $i$ be minimal such that 
\[
r_i/|t_i| = b\delta_i \leq \tau.
\]
If $i$ is even, then we have $i=j=2h+2$ in \eqref{eq:yv_q}. Then we are in the Case 1 from Proposition \ref{prp:shortbases}, and 
\eq\label{xy_formulas_jeven}
\begin{array}{l}
\xv = \langle q_{2h+2}, bq_{2h+2}\delta_{2h+2}\rangle = \left\langle q_{2h+2},bq_{2h+2}\left|\frac{a}{b}-\frac{p_{2h+2}}{q_{2h+2}}\right| \right\rangle, \\
\yv = \langle q_{2h+1}, -bq_{2h+1}\delta_{2h+1}\rangle = \left\langle q_{2h+1},-bq_{2h+1}\left|\frac{a}{b}-\frac{p_{2h+1}}{q_{2h+1}}\right| \right\rangle.
\end{array}
\eeq
If $i$ is odd, then we are in the Case 2 from Proposition \ref{prp:shortbases}, and we have
\eq\label{xy_formulas_jodd}
\begin{array}{l}
\xv = \langle q_{2h}, bq_{2h}\delta_{2h}\rangle = \left\langle q_{2h},bq_{2h}\left|\frac{a}{b}-\frac{p_{2h}}{q_{2h}}\right| \right\rangle, \\
\yv = \langle q_{2h+1}, -bq_{2h+1}\delta_{2h+1}\rangle = \left\langle q_{2h+1},-bq_{2h+1}\left|\frac{a}{b}-\frac{p_{2h+1}}{q_{2h+1}}\right| \right\rangle.
\end{array}
\eeq

\section{The shape of the permutation $w(n,a/N)$}\label{sec:fulllattice}

In this section we consider the shape of Young diagrams for \sos{} permutations $w=w(n,a/N)$, where $a, n$ and $N$ are non-negative integers. As mentioned, $\sh(w)=\sh(w^{-1})$, and upon vertical rescaling, we see the set
\[
\{(i, w^{-1}(i))\}_{i=1}^n
\]
(a subset of integer points in $[1,n]\times [1,n]$), the set
\[
\{(i, f_{a/N}(i) )\}_{i=1}^n 
\]
(a subset of points in $[1,n]\times [0,1)$), and the set
\[
\{ (i, a\cdot i \mod N) \}_{i=1}^n 
\]
(a subset of points in $[1,n]\times [1,N]$), all have the same order structure. Thus we can leverage all the tools developed in Section \ref{sec:paths} to study $\sh(w)$. It is convenient to identify the permutation with the third lattice, scaled vertically to fit in a square of size $n\times n$. Throughout this section, define the lattices
\[
L:=L_{a,N}=\{(i,a\cdot i) \mod N : i \in \Z\} \quad \textrm{ and } \quad L_n:=\left\{\left(x,\frac{n}{N}y\right) : (x,y) \in L_{a,N}\right\}.
\]
These lattices differ only by a scaling factor of $n/N$ in the vertical direction. Notice then that lines of slope $\pm\tau=\pm N/n$ in the lattice $L$ are mapped to lines of slope $\pm 1$ in the lattice $L_n$.
We define also
\[
L_n^{\boxslash}:=L_n\cap [0,n]^2 \quad \textrm{ and } \quad L_n^{\boxbslash}:=L_n\cap \left([0,n]\times [-n,0]\right).
\]
In summary, the shape $\lam=\sh(w)$ is characterized by increasing paths in $L_n^{\boxslash}$ and decreasing paths in $L_n^{\boxbslash}$.

In Proposition \ref{prp:shortbases}, it is shown that there are unit lattice vectors $\xv$ and $\yv$ forming a basis for $L$ whose slopes satisfy one of several inequalities, relative to a fixed slope $\tau$. Take $\tau = N/n$, and let vectors $\xv_n$ and $\yv_n$ be obtained from the vectors $\xv$ and $\yv$ in Proposition \ref{prp:shortbases} by $\xv_n = \langle 1, n/N\rangle \circ \xv$, $\yv_n =  \langle 1, n/N\rangle \circ  \yv$, where $\circ$ is the Hadamard (entrywise) product. The slopes of these rescaled basis vectors satisfy one of the following strings of inequalities:
\begin{align}
&m_{\yv_n} \le -1 \le m_{\yv_n+\xv_n} < 0 < m_{\xv_n} \le 1 \le m_{\xv_n-s\yv_n}, \label{avbv1a}\\
&m_{\yv_n} \le -1 \le m_{\yv_n+s\xv_n} < 0 < m_{\xv_n} \le 1 \le m_{\xv_n-\yv_n}, \label{avbv1b}\\
&m_{\yv_n-\xv_n} \le -1 \le m_{\yv_n} < 0 < m_{\xv_n+s\yv_n} \le 1 \le m_{\xv_n}, \label{avbv2a} \\
&m_{\yv_n-s\xv_n} \le -1 \le m_{\yv_n} < 0 < m_{\yv_n+\xv_n} \le 1 \le m_{\xv_n}, \label{avbv2b}
\end{align}
where, in each case, $s\geq 1$ is the minimal positive integer that makes the inequalities true. The four cases above correspond to the cases 1(a), 1(b), 2(a), and 2(b), respectively, from Proposition \ref{prp:shortbases}.  Recall from Section \ref{xy_conv} that $\xv$ and $\yv$ are determined by the minimal value of $i$ such that $N\delta_i \leq \tau$, or 
\[
\delta_i = \left|\frac{a}{N} - \frac{p_i}{q_i}\right| \leq \frac{\tau}{N} = \frac{1}{n},
\]
where $p_i/q_i$ is the $i$th convergent for $a/N$.

A rescaling of the formulas \eqref{xy_formulas_jeven} and \eqref{xy_formulas_jodd} gives
\eq\label{xyn_formulas_jodd}
\xv_n = \left\langle q_{2h+2}, nq_{2h+2}\left|\frac{a}{N} - \frac{p_{2h+2}}{q_{2h+2}}\right| \right\rangle, \qquad \yv_n = \left\langle q_{2h+1}, -nq_{2h+1}\left|\frac{a}{N} - \frac{p_{2h+1}}{q_{2h+1}}\right| \right\rangle,
\eeq
if $i=2h+2$ is even; and
\eq\label{xyn_formulas_jeven}
\xv_n = \left\langle q_{2h}, nq_{2h}\left|\frac{a}{N} - \frac{p_{2h}}{q_{2h}}\right| \right\rangle, \qquad \yv_n = \left\langle q_{2h+1}, -nq_{2h+1}\left|\frac{a}{N} - \frac{p_{2h+1}}{q_{2h+1}}\right| \right\rangle.
\eeq
if $i=2h+1$ is odd.

\subsection{Symmetries between cases}\label{sec:symmetry}

There are two fundamental lattice transformations that allow us to reduce our analysis of a priori eight cases (increasing/decreasing paths for each of Inequalities \eqref{avbv1a}--\eqref{avbv2b}) down to only two cases. In terms of the basis pair $(\xv, \yv)$, these two transformations are:
\[
 \left( \begin{array}{cc} \xv & \yv\end{array} \right) = \left( \begin{array}{cc} x_1 & y_1 \\ x_2 & y_2 \end{array} \right) \mapsto \left( \begin{array}{cc} -y_2 & x_2 \\ y_1 & -x_1 \end{array} \right) = \left( \begin{array}{cc} \xv^* & \yv^* \end{array} \right) = \rho\left( \begin{array}{cc} \xv & \yv\end{array} \right),
\]
and:
\[
 \left( \begin{array}{cc} \xv & \yv\end{array} \right) = \left( \begin{array}{cc} x_1 & y_1 \\ x_2 & y_2 \end{array} \right) \mapsto \left( \begin{array}{cc} y_1 & x_1 \\ -y_2 & -x_2 \end{array} \right) = \left( \begin{array}{cc} \overline{\xv} & \overline{\yv} \end{array} \right) = \omega\left( \begin{array}{cc} \xv & \yv\end{array} \right).
\]
It is easily verified that $\rho^2 = \omega^2$ is the identity map. Simple calculations yield the following observation.

\begin{obs}
Let $(\xv, \yv)$ be a lattice basis pair. Then:
\begin{enumerate}
\item $(\xv,\yv)$ satisfies inequalities \eqref{avbv1a} if and only if $(\xv^*,\yv^*)$ satisfies \eqref{avbv1b} if and only if $(\overline{\xv},\overline{\yv})$ satisfies \eqref{avbv2b}, and
\item $(\xv,\yv)$ satisfies inequalities \eqref{avbv2a} if and only if $(\xv^*,\yv^*)$ satisfies \eqref{avbv2b} if and only if $(\overline{\xv},\overline{\yv})$ satisfies \eqref{avbv1b}.
\end{enumerate}
\end{obs}

Moreover, both transformations $\rho$ and $\omega$ map increasing paths to decreasing paths and vice-versa. Therefore we can translate results for increasing paths in case \eqref{avbv1a} to results about decreasing paths in case \eqref{avbv1b}, and so on. A conceptual diagram for our cases is as follows:
\[
 (\textrm{Case } \eqref{avbv1a} \nearrow) \xleftrightarrow{\rho} (\textrm{Case } \eqref{avbv1b}\searrow) \xleftrightarrow{\omega} (\textrm{Case } \eqref{avbv2a}\nearrow) \xleftrightarrow{\rho} (\textrm{Case } \eqref{avbv2b}\searrow)
\]
and
\[
 (\textrm{Case } \eqref{avbv1b} \nearrow) \xleftrightarrow{\rho} (\textrm{Case } \eqref{avbv1a}\searrow) \xleftrightarrow{\omega} (\textrm{Case } \eqref{avbv2b}\nearrow) \xleftrightarrow{\rho} (\textrm{Case } \eqref{avbv2a}\searrow).
\]

Thus, it suffices to only verify our results in detail for increasing paths in cases \eqref{avbv1a} and \eqref{avbv1b}, with results translated via $\rho$ and $\omega$ to all other cases. In fact, the differences in the arguments used for case \eqref{avbv1b} versus those in \eqref{avbv1a} are negligible (all of Subsection \ref{sec:my} applies equally to both cases), so we focus on \eqref{avbv1a} for the duration of this section and leave the modifications for \eqref{avbv1b} to the reader.

With this assumption now fixed, and with a slight abuse of notation to streamline notation, we write our basis vectors as
\[
\xv=\xv_n = \langle x_1, x_2\rangle, \qquad \yv =\yv_n = \langle y_1, y_2\rangle,
\]
and assume the inequalities of case \eqref{avbv1a}. In particular, Subsection \ref{sec:lem} and Subsection \ref{sec:greene} make repeated use of the fact that $1\geq m_{\xv} > 0 > m_{\xv+\yv} \geq -1 \geq m_{\yv}$. 

\subsection{Lines of slope $m_{\yv}$}\label{sec:my}

Let $f_j$ be the line of slope $m_{\yv}$ passing though the point $j \xv$. That is, $f_j=\{t\yv + j\xv: t \in \R\}$. Points $(x,y) \in f_j$ satisfy 
\eq\label{eq:Lj_eq}
y=\frac{y_2}{y_1} x + j\left(x_2-\frac{y_2}{y_1} x_1\right)=\frac{y_2}{y_1} x+\frac{nj}{y_1},
\eeq
where we have used the identity $n=y_1x_2 - x_1y_2$ for the second equality.
Similarly, let $g_i$ be the line of slope $m_{\xv}$ passing through the point $i\yv$. These lines are described by the equations
\eq\label{eq:gi}
y=\frac{x_2}{x_1} x-\frac{ni}{x_1}.
\eeq

In what follows, our convention will be to treat $j$ as a real-valued parameter, and $i$ will typically be integer-valued. With that in mind, define the \emph{crossing set}
\[
\li_j := \{f_j \cap g_i : i\in \Z\} \cap [0,n]^2.
\]
If $j$ is an integer, then $f_j$ and $g_i$ necessarily intersect at lattice points, in which case $\li_j = f_j \cap L_n^{\boxslash}$.  We denote the number of crossings by $l_j=|\li_j|$.

As a function of $j$, $l_j$ is a step function which changes only by jumps of magnitude 1 or 2. More concretely, we see $l_j$ jumps when one or more of the following quantities are integers:
\begin{equation}\label{eq:fjgi}
 \frac{x_1}{y_1}j, \quad \frac{x_2}{y_2}j, \quad  \frac{x_1}{y_1}(y_2+j) - x_2, \quad \mbox{ or } \quad \frac{x_2}{y_2}(y_1-j) - x_1.
\end{equation}
These four cases correspond to a line $f_j$, $j\geq 0$, meeting a line $g_i$, $i\in \Z$, along the left edge, bottom edge, right edge, and top edge, respectively, of the box $[0,n]^2$. Since each line $f_j$, intersects the boundary of the box $[0,n]^2$ in at most two points, we see that there are at most two distinct integers among the values in \eqref{eq:fjgi}.

Suppose $j\in \N$. By Corollary \ref{cor:simpleangle}, we know that lattice vectors $\vv$ in $\li_j$  have $\ell^+(\vv) \leq j$, and that this bound is sharp. In particular, $\li_{\ell^+}$ contains the lattice vectors of maximal length, where we recall $\ell^+ = \arm(w)$ denotes the longest increasing lattice walk in $L_n$, and $\ell^- = \leg(w)$ denotes the longest decreasing walk. 

We can now find bounds for $\ell^+$ and, with a slight reorientation, $\ell^-$. 

\begin{prop}[Longest increasing and decreasing paths]\label{prp:l+bounds}
We have the following bounds on $\ell^+$ and $\ell^-$:
\eq\label{eq:l+bounds}
 y_1 - y_2 -2 < \ell^+ \leq y_1 - y_2, \quad \mbox{ and } \quad x_1+x_2 -2< \ell^- \leq x_1 + x_2.
\eeq
\end{prop}

By substituting our expressions for $x_1, x_2, y_1, y_2$ in terms of continued fraction convergents from Equations \eqref{xyn_formulas_jodd} and \eqref{xyn_formulas_jeven}, Proposition \ref{prp:l+bounds} yields Theorem \ref{thm:BSarm} for rational numbers $\alpha = a/N$.

\begin{proof}
Notice that if we solve for $j$ such that $\la n, n\ra \in f_j$, we find, from \eqref{eq:Lj_eq},
\[
 n = \frac{y_2}{y_1}n + \frac{j}{y_1}n,
\]
and thus $j=y_1-y_2$. Since $y_1-y_2$ is a priori rational, we see the lines with integer index intersecting $[0,n]^2$ are $f_0, f_1, \ldots, f_{\lfloor y_1 - y_2\rfloor}$. Thus $\ell^+ \leq y_1-y_2$. 

It is possible that this last line $f=f_{\lfloor y_1 - y_2\rfloor}$ has no lattice points inside $[0,n]^2$. If so, let $\xv'$ be the lowest point on $f$ above $[0,n]^2$ and let $\xv''=\xv'+\yv$ be the leftmost point on $f$ to the right of $[0,n]^2$. Note that since $f$ intersects $[0,n]^2$, $x'_2 \geq n \geq x_1'$, while $x_2''\leq n \leq x_1''$. 

The vector $\xv-\yv$ has positive slope, and so $\xv'''=\xv'-(\xv-\yv)=\xv''-\xv$ is to the left of $\xv'$ (so $x'''_1 < n$) and below $\xv''$ (so $x'''_2 < n$). In other words, $\xv''' \in \li_{\lfloor y_1 - y_2 \rfloor -1}\neq \emptyset$. This shows $\ell^+$ is at least $\lfloor y_1 - y_2 \rfloor -1$, i.e., $y_1 - y_2 - 2 < \ell^+$, which proves the desired result.

The case of $\ell^-$ follows the same argument for increasing paths in the lattice with basis $\xv^* = \langle -y_2,y_1\rangle$ and $\yv^* = \langle x_2, -x_1\rangle$, where we find $y_1^*-y_2^* = x_2+x_1$.
\end{proof}

For example, in Figure \ref{fig:ablatticelength}, we have $a = 51, n=N=71$, and the upper bound is realized. There, $\xv=\langle 7,2\rangle$ and $\yv = \langle 4,-9\rangle$ and we see $\ell^+ = 4+9 = 13$, while $\ell^-=7+2 = 9$.

\begin{remark}
One may compare the bounds in Proposition \ref{prp:l+bounds} with Corollary 1 in \cite{BoydSteele}. The quantities $y_1-y_2$ and $x_1+x_2$, denoted $\lam_n^+$ and $\lam_n^-$ in \cite{BoydSteele}, are the solutions to the linear programming problem in that paper, which tightly approximate $\ell^+$ and $\ell^-$ (denoted $l_n^+$ and $l_n^-$ in that paper), the solutions to an integer  programming problem.
\end{remark}

In what follows, it will be convenient to partition the collection of points in $L_n^{\boxslash}$ according to where they intersect the lines $f_j$, and according to where those lines intersect the boundary of $[0,n]^2$. Define
\[
 \Lcal = \bigcup_{0\leq j \leq y_1} \li_{\lfloor j \rfloor}, \quad 
 \Mcal = \bigcup_{y_1 < j \leq  -y_2 } \li_{\lfloor j \rfloor}, \mbox{ and } \quad 
 \Rcal = \bigcup_{  -y_2< j \leq y_1-y_2} \li_{\lfloor j \rfloor},
\]
so that $\Lcal$ contains points from lines $f_j$ that intersect $[0,n]^2$ on the left and bottom edges, $\Mcal$ contains points from lines that intersect on the top and bottom, while set $\Rcal$ contains points from lines that intersect on the top and right edges. 

Notice that, as a function of $j$, $l_j$ is weakly increasing for $0\le j \le y_1$ and weakly decreasing for $|y_2| \le j \le y_1-y_2$. In what follows, for a fixed positive integer $k$, we will need to identify the first and last lines $f_j$ which contain at least $k$ crossings. With that in mind, define for $k=1,2,3,\ldots, l_{y_1}$,
\[
j_k := \min_{0\le j \le y_1}\{ j : l_j \ge k\}, \qquad j_k' := \max_{|y_2| \le j \le y_1- y_2}\{ j : l_j \ge k\}.
\]

We now present some straightforward results for counting both the number of crossings $l_j=|\li_j|$, as well as estimating the value of $j_k$ and $j_k'$.

\begin{prop}[Counting crossings and lines]\label{prop_lj}
The numbers $l_j$ satisfy the bounds
\begin{align}
\frac{jn}{|y_1y_2|} -1 &\le l_j \le \frac{jn}{|y_1y_2|} +1, \quad &&\textrm{for} \ 0\leq j\leq y_1, \label{eq:ljbound1} \\
\frac{n}{|y_2|} -1 &\le l_j \le \frac{n}{|y_2|} +1, \quad &&\textrm{for} \ y_1 < j \leq |y_2|, \label{eq:ljbound2} \\
\frac{n(y_1-y_2-j)}{|y_1y_2|} -1 &\le l_j \le \frac{n(y_1-y_2-j)}{|y_1y_2|} +1, \quad &&\textrm{for} \ |y_2|< j \le y_1-y_2.\label{eq:ljbound3}
\end{align}
Thus, for any $k=1,2,\ldots,l_{y_1}$ we have 
\begin{align}
\frac{(k-1)|y_1y_2|}{n}  &\le j_k \le \frac{(k+1)|y_1y_2|}{n}, \label{eq:jk_est} \\
y_1-y_2-\frac{(k+1)|y_1y_2|}{n} &\le j_k' \le y_1-y_2-\frac{(k-1)|y_1y_2|}{n}.\label{eq:jkp_est}
\end{align}
\end{prop} 

\begin{proof}
The inequalities \eqref{eq:ljbound1}--\eqref{eq:ljbound3} are simply given by the Euclidean length of line segment $f_j \cap [0,n]^2$ divided by $|\yv|$, plus or minus 1.

The inequality \eqref{eq:jk_est} follows from \eqref{eq:ljbound1} along with the fact that $j_k$ satisifies $l_{j_k-\epsilon}< k \leq l_{j_k}$ for any $\epsilon >0$. Similarly \eqref{eq:jkp_est} follows from \eqref{eq:ljbound3} along with the fact that $l_{j_k'+\epsilon}< k \leq l_{j_k'}$ for any $\epsilon >0$.
\end{proof}

\subsection{Lemmas for lines of slope $m_{\yv}$}\label{sec:lem}

In this subsection, we record two useful lemmas regarding the lines of slope $m_{\yv}$. We first compare crossings at a fixed distance from the bottom left and top right corner. Define 
\[
\ol{j} := y_1-y_2-j.
\]

\begin{lem}\label{lem:ldiff}
For any $j \leq y_1$,
\[
 |l_j - l_{\ol{j}}|\leq 1.
\]
\end{lem}

\begin{proof}
We can be completely explicit here. The line $g_{x_2-x_1}$ passes through $\langle n, n\rangle$. If $i=x_2-x_1 \in \Z$, then the lines of slope $m_{\xv}$ passing through $[0,n]^2$ have 180 degree rotational symmetry, and we have $l_j = l_{\ol{j}}$ for all $j$. 

In general, however, $i\notin \Z$, and $g_{\lfloor i \rfloor}$ and $g_{\lceil i \rceil}$ are the nearest integer lines to this top corner. Let $A$ denote the horizontal spacing between lines of slope $m_{\xv}$, and let $B$ denote the vertical spacing. Then from Equation \eqref{eq:gi} we deduce
\[
 A = \frac{n}{x_2} \quad \textrm{ and } \quad B = \frac{n}{x_1}.
\]
Let $a$ denote the horizontal distance from $g_{\lfloor i \rfloor}$ to $\langle n, n\rangle$ along the line $y=n$, and let $b$ denote the vertical distance from $g_{\lceil i \rceil}$ to $\langle n,n \rangle$ along the line $x=n$. Then solving for these values concretely (again using \eqref{eq:gi}), we have:
\[
a = \frac{n}{x_2}\left( (x_2-x_1) - \lfloor x_2 - x_1\rfloor \right) \quad \textrm{ and } \quad b = \frac{n}{x_1}\left( \lceil x_2 - x_1\rceil - (x_2-x_1) \right).
\]

For any $j \in \Q$, write $w(j) = -\frac{n}{y_2}j$ and $h(j) = \frac{n}{y_1}j$ to denote the width and height of the triangle formed by the origin and the $x-$ and $y-$intercepts of the line $f_j$. Note that by definition, $w(j)$ (resp. $h(j)$) is also the horizontal (resp. vertical) distance between $f_{\ol{j}}$ and the corner $\langle n, n\rangle$. 

For $j < y_1$, define $l^{\lf}_j$ to be the number of integer-indexed lines $g_i$ that intersect the line $x=0$ with a positive vertical component at most $h(j)$ and define $l^{\bo}_j$ to be the number of lines $g_i$ that intersect the line $y=0$ with a positive horizontal component at most $w(j)$. Then counting the line $g_0$ that passes through the origin, we have
\[
 l_j = l^{\lf}_j + l^{\bo}_j + 1.
\]

Similarly, let $l^{\tp}_j$ denote the number of lines $g_i$ that intersect the line $y=n$ at a horizontal component $w(j)\leq x < n$, and let $l^{\rt}_j$ denote the number of lines that intersect the line $x=n$ at height $h(j) \leq y < n$. Then (unless $x_2 - x_1 \in \Z$ as discussed earlier), we have
\[
 l_{\ol{j}} = l^{\tp}_j + l^{\rt}_j.
\]

Now suppose $w(j) = mA+r$ for some $0\leq r < A$. Then 
\[
 l^{\bo}_j = m \quad \textrm{ and } \quad l^{\tp}_j = m+ \begin{cases} 0 & \textrm{ if } r < a, \\ 1 & \textrm{ if } a \leq r < A. \end{cases}
\]
Similarly if $h(j) = m'B+r'$,
\[
 l^{\lf}_j = m' \quad \textrm{ and } \quad l^{\rt}_j = m'+ \begin{cases} 0 & \textrm{ if } r' < b, \\ 1 & \textrm{ if } b \leq r' < B. \end{cases}
\]

Summing, the result follows since $l^{\bo}_j + l^{\lf}_j = l_j - 1$, and thus,
\[
l_{\ol{j}} = l^{\tp}_j + l_j^{\rt} = l_j + \begin{cases} -1 & \textrm{ if } r < a, r'<b \\ 1 & \textrm{ if } a \leq r < A, b \leq r' < B, \\ 0 & \textrm{ otherwise.} \end{cases}
\]
\end{proof}

We now introduce notation to indicate mapping an arbitrary crossing set $\li_j$ into an integer-indexed crossing set. For $j < y_1$, we define
\[
 \psi_j = \{ \wv + (\lceil j \rceil - j) \xv : \wv \in \li_j\} \subseteq \li_{\lceil j \rceil},
\]
and for $j > |y_2|$, we define
\[
 \phi_j = \{ \wv - (j-\lfloor j \rfloor) \xv : \wv \in \li_j\} \subseteq \li_{\lfloor j \rfloor}.
\]
By construction $\psi_j, \phi_j \in L_n^{\boxslash}$.

\begin{lem}\label{lem:bot}
Suppose $j < y_1$. If $\zv$ is the $i$th lattice point from the top in $\psi_j$, then $\zv - (i+1)\xv \notin [0,n]^2$. Simlilarly, if $j > |y_2|$ and $\zv$ is the $i$th lattice point from the bottom in $\phi_j$, then $\zv + (i+1)\xv \notin [0,n]^2$.
\end{lem}

\begin{proof}
We argue only for the case of $j<y_1$, as the case of $j>|y_2|$ follows by 180 degree rotation.

First, we make the simple observation that because $\xv -\yv$ is in the first quadrant, $x_1 > y_1$. Suppose $\zv'=\langle z_1, z_2\rangle$ is the $i$th crossing from the top in $\li_j$. Then $\zv'-i\yv \notin [0,n]$, and $z_1-iy_1 < 0$. But then $z_1 - ix_1 < z_1 - iy_1 <0$ as well, and hence $\zv' - i \xv \notin [0,n]^2$.

Now, $\zv = \zv'+\delta\xv$, where $\delta = (\lceil j \rceil - j) <1$. Thus $\zv - (i+1) \xv = \zv' - i\xv - (1-\delta)\xv$, which is clearly outside of $[0,n]^2$.
\end{proof}

We now define the quantity
\eq\label{eq:j_bound}
J_0=  \frac{y_1y_2(x_1-x_2)}{x_1y_2+x_2y_1}.
\eeq
which will play a role in much of what follows. 

\begin{lem}\label{lem:pathP}
Fix $k$ such that $j= \max\{j_k, \overline{j_k'}\} \leq J_0$. Let $\wv$ be the topmost point in $\li_j$ and let $\vv=\wv+(\lceil j \rceil - j)\xv$ be the topmost point in $\psi_j$. Similarly, let $\wv'$ be the topmost point in $\li_{\overline{j}}$ and let $\vv'= \wv'-(\ol{j} - \lfloor \ol{j} \rfloor)\xv$ be the topmost point in $\phi_{\ol{j}}$.  Then there is an increasing path of length $\lfloor \overline{j}\rfloor - \lceil j \rceil$ from $\vv$ to $\vv'$. 
\end{lem}

\begin{proof}
For each $j$ and $\vv$, $\vv'$ defined as above, let $m_j$ denote the slope of the line from $\vv$ to $\vv'$. Corollary \ref{cor:simpleangle} says there is an increasing path of length $\lfloor \overline{j} \rfloor-\lceil j\rceil $ from $\vv$ to $\vv'$ if $m_{\xv} \leq m_j \leq m_{\uv}$.

We first show $m_j \ge m_{\xv}$. It will be convenient to consider the slope of the line from $\wv$ to $\wv'$, which we denote $\tilde{m}_j$. We will show that $\tilde m_j \ge m_{\xv}$, which is sufficient to show $m_j \ge \tilde{m}_{j}$.  Indeed, if $\tilde{m}_j \ge m_{\xv}$, then the point $\vv = \wv+(\lceil j \rceil - j)\xv$ is weakly below the line segment connecting $\wv$ and $\wv'$, and $\vv'=\wv' -(\ol{j} - \lfloor \ol{j} \rfloor)\xv$ is weakly above it. Thus $m_j \ge \tilde{m}_j \ge m_{\xv}$. So let us show that $\tilde m_j \ge m_{\xv}$. Since $\wv+(\overline{j} - j)\xv$ is in $f_{\overline{j}}$, it is enough to show that $\wv+(\overline{j} - j)\xv$ has vertical component at most $n$.

 Note that $\wv$ is very near the point $\langle 0, \frac{nj}{y_1}\rangle$, the point at which line $f_j$ intersects the left edge of $[0,n]^2$. In particular, $ w_2 \le \frac{nj}{y_1}$. 
Computing this vertical component  we find:
\begin{align*}
 w_2 + (\overline{j}-j)x_2 &< \frac{nj}{y_1} + (\overline{j}-j)x_2=\frac{nj}{y_1} + (\overline{j}+j - 2j)x_2 \\
  & \le  j\left( \frac{n}{y_1} - 2x_2\right) + (y_1-y_2)x_2.
\end{align*}
Using $j\leq \frac{y_1y_2(x_1-x_2)}{x_1y_2+x_2y_1}$ we find 
\begin{align*}
j\left( \frac{n}{y_1} - 2x_2\right) + (y_1-y_2)x_2 &\leq \frac{y_1y_2(x_1-x_2)}{x_1y_2+x_2y_1}\left( \frac{n}{y_1} - 2x_2\right) + (y_1-y_2)x_2 \\
&= \frac{y_2(x_1-x_2)(n-2x_2y_1)}{x_1y_2+x_2y_1} + (y_1-y_2)x_2 \\
&= -y_2(x_1-x_2) + (y_1-y_2)x_2\\
&= -y_2x_1+y_1x_2 = n,
\end{align*}
where we have used $n = -y_2x_1+y_1x_2$ in two different places. Thus $\wv+(\overline{j}-j)\xv$ has vertical component at most $n$, and $\tilde m_j \geq m_{\xv}$ as desired.

To show $m_j \leq m_{\uv}$, we use a more direct comparison of the two slopes. The line $f_j$ hits the left edge of the box $[0,n]^2$ at height $nj/y_1$, so $\vv$ can be no lower than $nj/y_1 + y_2$. On the other hand, $\vv$ can be no farther to the right than $x_1$, since otherwise $\vv-\xv$ would still be in the box. With this in mind, we denote $A=(x_1, nj/y_1 + y_2)$, which is weakly below and to the right of point $\vv$.

Similarly, we find a point $B$ which is weakly above and left of $\vv'$. The point $\vv'$ is no higher than $n$ itself. The line $f_{\ol{j}}$ intersects the top edge of the box at horizontal component $-n(\ol{j} - y_1)/y_2$, and so point $\vv'$ can be no farther left than $-n(\ol{j} - y_1)/y_2 - x_1$. So we can take $B$ to be the point $B=(-n(\ol{j} - y_1)/y_2 - x_1,n)$.

Denoting the slope of the line between the points $A$ and $B$ as $m_{AB}$, we see that $m_j \le m_{AB}$, so it suffices to show $m_{AB} \leq m_{\uv}$.
We can write the slope $m_{AB}$ explicitly:
\begin{align*}
 m_{AB} &= \frac{ n -  nj/y_1 - y_2}{-n(\ol{j} - y_1)/y_2 - x_1 - x_1}\\
  &= \frac{ny_2(y_1-j) - y_2^2y_1}{-ny_1(\ol{j} - y_1) -2x_1y_1y_2}\\
  &= \frac{-ny_2(y_1-j) + y_2^2y_1}{ny_1(-y_2-j) + 2x_1y_1y_2}.
\end{align*}


We now show that $m_{AB} < u_2/u_1=m_{\uv}$ for all $0<j\le y_1$. Since $J_0\le y_1$, this will complete the proof.  Let
\begin{align*}
N(j) &= u_1(-ny_2(y_1-j) + y_2^2y_1) = j\cdot ny_2u_1 -ny_2y_1u_1 + y_2^2y_1, \\
D(j) &= u_2(ny_1(-y_2-j) + 2x_1y_1y_2) =  -j\cdot ny_1u_2 - ny_2y_1u_2 + 2x_1y_2y_1.
\end{align*}
Then $m_{AB} < u_2/u_1$ precisely when $N(j) < D(j)$. 
Solving the inequality $N(j) < D(j)$ for $j$ yields
\begin{equation}\label{eq:jb1}
 j < \frac{ -y_2y_1( u_2 - u_1 +2x_1/n-y_2/n)}{ y_2u_1+y_1u_2}.
\end{equation}
(Note that we have divided by $y_2u_1+y_1u_2$, which is necessarily positive since $u_2 > -y_2 >0$ and $y_1 > u_1 > 0$.) We are left to show that the right hand side of \eqref{eq:jb1} is bounded by $y_1$.

Since 
\[
\uv= \xv-a\yv, \quad a= \lfloor x_1/y_1 \rfloor,
\]
we can rewrite the right hand side of \eqref{eq:jb1} as
\begin{equation}\label{eq:jb2}
\frac{- y_2y_1( x_2 - x_1 +(2x_1-y_2)/n+a(y_1-y_2))}{ x_2y_1+y_2x_1-2ay_1y_2}.
\end{equation}
Consider this expression as a function of $a$ (we know that $a=\lfloor x_1/y_1 \rfloor$, but consider $a$ as free for now). There is a vertical asymptote at $a=x_1/(2y_1)+x_2/(2y_2)< \lfloor x_1/y_1 \rfloor$ (recall that $x_1>y_1$). To the right of this asymptote the denominator is clearly positive, since it is positive for large enough $a$. (In fact, at $a=x_1/y_1$, the denominator is $x_2y_1-y_2x_1 = n$.) We can check that the numerator is positive as well by plugging in $a=x_1/(2y_1)+x_2/(2y_2)$:
\begin{multline*}
- y_2y_1( x_2 - x_1 +(2x_1-y_2)/n+(x_1/(2y_1)+x_2/(2y_2))(y_1-y_2)) = \\
 - y_2y_1\left[\frac{x_1}{2}\left(-\frac{y_2}{y_1} -1 + \frac{4}{n}\right)+\frac{x_2(y_1+y_2)}{2y_2} - \frac{y_2}{n}\right].
\end{multline*}
Since $0<y_1\le-y_2$, each of the three terms inside the brackets is non-negative, and the first and last terms are positive, so the entire expression is as well.

Since both the numerator and denominator of \eqref{eq:jb2} are positive for $a$ immediately to the right of the vertical asymptote, the expression is decreasing in $a>x_1/(2y_1)+x_2/(2y_2)$. It follows that the value of \eqref{eq:jb2} at $a=\lfloor x_1/y_1 \rfloor>x_1/(2y_1)+x_2/(2y_2)$ is greater than the value in the limit as $a\to\infty$, which is
\[
\lim_{a\to +\infty} \frac{- y_2y_1a(y_1-y_2)}{-2ay_1y_2} = \frac{y_1-y_2}{2}  \ge y_1.
\]
Thus $N(j) < D(j)$, and therefore $m_{AB} < m_{\uv}$, for all $0<j \le y_1$.
\end{proof}

\subsection{Constructive use of Greene's Theorem}\label{sec:greene}

Recall $I_k=\lambda_1+\cdots+\lambda_k$ denotes the size of the largest subsequence formed by the union of $k$ increasing subsequences. We will prove a nearly exact formula for $I_k$ in terms of counting lines of slope $m_{\yv}$, which in turn will give us a uniform estimate for $\lambda_k$.

First, we make the easy observation that any union of $k$ increasing subsequences can intersect each crossing set $\li_i$ in at most $k$ points. Otherwise, since $\yv$ is negatively sloped, we would have a decreasing run of size $k+1$. Making this observation for each integer $i$ yields the following upper bound for $I_k$.

\begin{obs}\label{obs:up_bd_Ik}
For any $k$,
 $$I_k \leq \sum_{i=1}^{\lfloor y_1-y_2\rfloor} \min( l_i, k).$$
\end{obs}

We will prove that the above inequality is nearly an equality when $k$ is not too big by constructing a collection of $k$ increasing paths in $L_n^{\boxslash}$. Here is the general idea of the construction. Fix $j = \max\{j_k, \ol{j_k'}\}$ as in Lemma \ref{lem:pathP}, and let $P$ denote the path from the topmost lattice point in $\psi_j$ to the topmost lattice point in $\phi_{\ol{j}}$. 

Define set $S=S(j)$ to be the union of the following four sets:
\begin{itemize}
\item $\Lcal_j$, containing of all points in $\li_1, \li_2, \ldots, \li_{\lceil j \rceil -1}$,
\item $\Mcal_j$, containing all points in paths $P, P+\yv, \ldots, P+(k-1)\yv$, (these connect the top $k$ points in $\psi_j$ with the top $k$ points in $\phi_{\ol{j}}$)
\item $\Rcal_j$, containing all points in $\li_{\lfloor \ol{j} \rfloor +3}, \li_{\lfloor \ol{j} \rfloor +4}, \ldots, \li_{\lfloor y_1 - y_2\rfloor}$, and
\item $X$, containing all lattice points in $\li_{\lfloor \ol{j} \rfloor +1}$ and $\li_{\lfloor \ol{j} \rfloor +2}$ of the form $\wv + \xv$ or $\wv + 2\xv$, where $\wv$ is one of the topmost $k$ points in $\phi_{\ol{j}}$.
\end{itemize}
See Figure \ref{fig:parallel}. We now show $S$ is a union of $k$ increasing paths of cardinality at least 3 less than the upper bound of Observation \ref{obs:up_bd_Ik}.

\begin{prop}\label{prp:construction}
For any $k\leq l_{J_0}$, let $j=\max\{j_k,\ol{j_k'}\}$. Then the set $S=S(j)$ defined above is a union of $k$ increasing paths. Moreover,$|S| \geq -3 + \sum_{i=1}^{\lfloor y_1-y_2\rfloor} \min( l_i, k).$
Thus, 
\[
-3 + \sum_{i=1}^{\lfloor y_1-y_2\rfloor} \min( l_i, k) \leq I_k \leq \sum_{i=1}^{\lfloor y_1-y_2\rfloor} \min( l_i, k).
\]
\end{prop}

\begin{proof}
Without loss of generality, we may assume $j_k \leq \ol{j_k'}$, so that $j=j_k$. The case of $\ol{j_k'} < j_k$ follows the same approach under 180 degree rotation. 

First, we see each lattice point in $\Lcal_j$ can be connected by some number of $\xv$ steps to one of the translates of path $P$ that make up $\Mcal_j$. This is obvious if $l_j = k$. If $l_j = k+1$ and $l_{j-\epsilon} < k$, then because $l_j$ is a step function whose steps are of size 1 or 2, it must be that there are crossings at both the points where $f_j$ intersects the boundary of $[0,n]^2$. Hence, $\psi_{j-\epsilon}$ is a subset of the middle $k-1$ points in $\psi_j$, which are all in $\Mcal_j$. 

Similarly, any point in sets $\Rcal_j$ or $X$ can be connected by some number of $-\xv$ steps to one of these paths at a point in $\phi_{\ol{j}}$. The points of $X$ are defined to be those points of the form $\wv +\xv$ and $\wv + 2\xv$, where $\wv$ is one of the top $k$ points in $\phi_{\ol{j}}$, so they are obviously part of the same $k$ increasing sets. The argument for $\Rcal_j$ requires some more care.

From Lemma \ref{lem:ldiff}, we have $|l_j - l_{\ol{j}}|\leq 1$, so because $l_j \leq k+1$, we have $l_{\ol{j}} \leq k+2$. From Lemma \ref{lem:bot} we know that if $\zv$ and $\zv+\yv$ are the two bottommost points in $\phi_{\ol{j}}$, then $\zv + 3\xv \notin [0,n]^2$, and $(\zv+\yv)+2\xv \notin  [0,n]^2$. Hence crossing set $\li_{\lfloor \ol{j} \rfloor+3}$ has at most $l_{\ol{j}} -2 \leq k$ points, and moreover, each of them can be connected by $-3\xv$ to one of the top $k$ points in $\phi_{\ol{j}}$. 

Thus, all points in $S$ are connected to one of the $k$ translates of path $P$ and $S$ is the union of $k$ increasing sets.

For the claimed lower bound on the cardinality of $S$, we observe 
\[
 |S| = |\Lcal_j|+|\Mcal_j|+|\Rcal_j|+|X|,
\]
where 
\[
|\Lcal_j| = \sum_{i=1}^{\lceil j \rceil -1} l_i, \quad |\Mcal_j| = \sum_{i=\lceil j \rceil}^{\lfloor \ol{j} \rfloor} k, \quad \mbox{ and } \quad |\Rcal_j| = \sum_{i=\lfloor \ol{j} \rfloor+3}^{\lfloor y_1-y_2 \rfloor} l_i.
\]

Finally, we observe that if $\vv$ is the topmost lattice point in $\phi_{\ol{j}}$, then 
\[
 \phi_{\ol{j}} \subseteq \{ \vv, \vv+\yv, \ldots, \vv+k\yv, \vv+(k+1)\yv \}.
\]
Therefore, $X = \li_{\lfloor \ol{j} \rfloor +1} \cup \li_{\lfloor \ol{j} \rfloor +2} - \{ \vv+k\yv + \xv,  \vv+(k+1)\yv + \xv, \vv+k\yv + 2\xv\}$, and so
\[
 |X| \geq \min(l_{\lfloor \ol{j} \rfloor +1}, k) - 2 + \min(l_{\lfloor \ol{j} \rfloor +2}, k)-1,
\]
which completes the proof.
\end{proof}

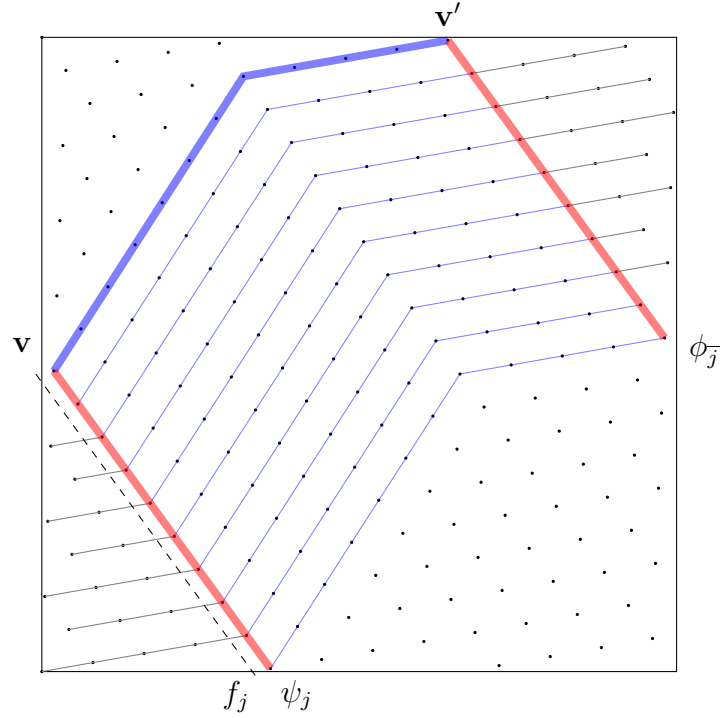
\begin{figure}
\begin{tikzpicture}[scale=.04, >=stealth]
\foreach \x in {0,...,210}{
 \draw[fill=black] (\x, {mod(\x*25,211)} ) circle (10pt);
}
\draw (0,211) circle (10pt);
\draw (0,211)--(0,0)--(211,0)--(211,211)--(0,211);
\draw[gray] (0,0)--(4*17,4*3);
\draw[gray] (17-8,3+11)--(4*17-8,4*3+11);
\draw[gray] (17-2*8,3+2*11)--(4*17-2*8,4*3+2*11);
\draw[gray] (2*17-3*8,2*3+3*11)--(4*17-3*8,4*3+3*11);
\draw[gray] (2*17-4*8,2*3+4*11)--(4*17-4*8,4*3+4*11);
\draw[gray] (3*17-5*8,3*3+5*11)--(4*17-5*8,4*3+5*11);
\draw[gray] (3*17-6*8,3*3+6*11)--(4*17-6*8,4*3+6*11);
\draw[gray] (211-17+2*8,211-3-2*11)--(211-4*17+2*8,211-4*3-2*11);
\draw[gray] (211-17+1*8,211-3-1*11)--(211-4*17+1*8,211-4*3-1*11);
\draw[gray] (211-17,211-3)--(211-4*17,211-4*3);
\draw[gray] (211-2*17+3*8,211-2*3-3*11)--(211-4*17+3*8,211-4*3-3*11);
\draw[gray] (211-2*17+4*8,211-2*3-4*11)--(211-4*17+4*8,211-4*3-4*11);
\draw[gray] (211-3*17+5*8,211-3*3-5*11)--(211-4*17+5*8,211-4*3-5*11);
\draw[gray] (211-3*17+6*8,211-3*3-6*11)--(211-4*17+6*8,211-4*3-6*11);
\draw[blue, line width=3,opacity=.5] (4*17-8*8,4*3+8*11)--(4*17-8*8+7*9,4*3+8*11+7*14)--(211-4*17-1*8, 211-4*3+1*11);
\draw[blue,opacity=.5] (4*17+1*8,4*3-1*11)--(4*17+1*8+7*9,4*3-1*11+7*14)--(211-4*17+8*8, 211-4*3-8*11);
\foreach \x in {1,...,8}{
\draw[blue,opacity=.5] (4*17+1*8-\x*8,4*3-1*11+\x*11)--(4*17+1*8+7*9-\x*8,4*3-1*11+7*14+\x*11)--(211-4*17+8*8-\x*8, 211-4*3-8*11+\x*11);
}
\draw[red,line width=3,opacity=.5] (4*17-8*8,4*3+8*11)--(4*17+1*8,4*3-1*11);
\draw[dashed] (3.65*17-8*8,3.65*3+8*11)--(3.65*17+1.3*8,3.65*3-1.3*11);
\draw (3.8*17+1*8,3.8*3-1*11) node[below left] {$f_j$};
\draw[red,line width=3,opacity=.5] (211-4*17+8*8,211-4*3-8*11)--(211-4*17-1*8,211-4*3+1*11);
\draw (4*17-8*8,4*3+8*11) node[xshift=-12pt,yshift=10pt, fill=white, inner sep=1] {$\vv$};
\draw (211-4*17-1*8,211-4*3+1*11) node[yshift=10pt, fill=white, inner sep=1] {$\vv'$};
\draw (4*17+1*8,4*3-1*11) node[yshift=-10pt,xshift=10pt, fill=white, inner sep=1] {$\psi_{j}$};
\draw (211-4*17+8*8, 211-4*3-8*11) node[xshift=15pt, yshift=-5pt, fill=white, inner sep=1] {$\phi_{\ol{j}}$};
\end{tikzpicture}
\caption{Example construction of set $S$. The blue path illustrates an increasing path $P$ constructed from $\vv$, the top point on $\psi_j$, to $\vv'$, the top point on $\phi_{\ol{j}}$. Here $a = 25$, $n=210$, and $N=211$. This yields $\xv=\langle 17,3\rangle$ and $\yv=\langle 8,-11\rangle$.}\label{fig:parallel}
\end{figure}

\subsection{Bounding the shape of $w$}\label{sec:shape-bounds}

Let $\lin(k) = \lfloor j_k' \rfloor- \lceil j_k \rceil +1$ denote the number of integers $i$ such that $l_i \geq k$. Since
\[
\lambda_k = I_k - I_{k-1}, 
\]
and since $\sum_i \min(l_i,k) - \sum_i \min(l_i,k-1) = \lin(k)$, we can translate Proposition \ref{prp:construction} into the following bound on $\lambda_k$, provided $k \leq l_{J_0}$:
\begin{equation}\label{eq:lamk}
 \left| \lambda_k - \lin(k) \right| \leq 3.
\end{equation}

The analogous result for $\lambda_k'$ follows by applying the transformation $\rho$ defined in in Subsection \ref{sec:symmetry}, which takes the pair $(\xv,\yv)$ satisfying \eqref{avbv1a} to the pair $(\xv^*,\yv^*) = (\langle -y_2,y_1\rangle, \langle x_2, -x_1\rangle)$ satisfying \eqref{avbv1b}. As discussed, each decreasing path in the lattice spanned by $(\xv,\yv)$ corresponds to an increasing path in the lattice spanned by $(\xv^*,\yv^*)$. 

Define
\[
J^*_0 =  \frac{y^*_1y^*_2(x^*_1-x^*_2)}{x^*_1y^*_2+x^*_2y^*_1} = \frac{x_1x_2(y_1+y_2)}{x_1y_2+x_2y_1},
\]
and let $l^*_j =|f^*_j \cap L^{\boxbslash}_n|$, where $f_j^* = j\xv^* + t\yv^*$ is the line of slope $m_{\yv^*}$ passing through $j\xv^*$.

Now, using the estimates in Proposition \ref{prop_lj} \eqref{eq:jk_est} and \eqref{eq:jkp_est}, we deduce
\[
\left|(y_1-y_2) - \frac{2|y_1y_2|}{n}k - \lin(k) \right| < \frac{2|y_1y_2|}{n}+1.
\]
With the characterization of $\yv$ in either \eqref{xyn_formulas_jodd} or \eqref{xyn_formulas_jeven}, we have
\[
\frac{|y_1y_2|}{n} = q_{2h+1}^2\delta_{2h+1}.  
\]
Recalling the Diophantine estimate for $\delta_{2h+1}$ in \eqref{eq:diophantine}, we have
\[
q_{2h+1}^2\delta_{2h+1} < \frac{q_{2h+1}^2}{q_{2h+2}q_{2h+1}} = \frac{q_{2h+1}}{q_{2h+2}} < 1.
\]
Thus, $\frac{2|y_1y_2|}{n}+1 < 3$.

We can now state simple bounds on $\lambda_k$, and by symmetry, $\lambda_k'$.

\begin{prop}\label{prp:twolines}
For $k \leq l_{J_0}$,
\[
 \left| (y_1-y_2)-\frac{2|y_1y_2|}{n}k - \lambda_k \right| < 4+\frac{2|y_1y_2|}{n} < 6,
\]
and for $k^* \leq l^*_{J^*_0}$,
\[
 \left| (x_1+x_2)-\frac{2x_1x_2}{n}k^* - \lambda'_{k^*} \right| < 4+\frac{2x_1x_2}{n} <6.
\]
\end{prop}

Proposition \ref{prp:twolines} establishes that there are two lines that roughly approximate the boundary of shape $A_{\lambda}$:
\eq\label{eq:xy-lines}
 y=(x_1+x_2)-\frac{2x_1x_2}{n}x \quad \mbox{ and } \quad  x=(y_1-y_2)-\frac{2|y_1y_2|}{n}y.
\eeq
One can verify (using $x_2y_1-y_2x_1=n$) that these two lines intersect at the point $(x_0,y_0)$ given by
\begin{equation}\label{eq:x0y0}
x_0 = \frac{n(y_1+y_2)}{x_1y_2+x_2y_1} \quad \mbox{ and } \quad y_0 = \frac{n(x_2-x_1)}{x_1y_2+x_2y_1}.
\end{equation}
Furthermore, from Proposition \ref{prop_lj} \eqref{eq:ljbound1} we see this point of intersection is very nearly $(l^*_{J^*_0}, l_{J_0})$ since 
\[
l^*_{J^*_0}-1 \leq  x_0 \leq l^*_{J^*_0}+1,
\]
and
\[
 l_{J_0}-1 \leq y_0 \leq l_{J_0}+1.
\]

Let $L(x)$ denote the piecewise linear function
\[
 L(x) =  
 \begin{cases} 
 (x_1+x_2)-\frac{2x_1x_2}{n}x & \mbox{for $0\leq x \leq x_0$,} \\
\frac{n(y_1-y_2-x)}{2|y_1y_2|} & \mbox{for $x_0 \leq x \leq y_1-y_2$},
\end{cases}
\]
and recall that $\partial A_{\lambda}$ denotes the boundary of $A_{\lambda}$ that does not lie on a coordinate axis. For fixed $x$, let $\dis( L(x), \partial A_{\lambda} )$ denote the minimum Euclidean distance from $(x, L(x))$ to a point on $\partial A_{\lambda}$. Proposition \ref{prp:twolines} gives us a uniform bound of 6 on this distance for all points $(x,y)$ with $x\leq x_0-1$ and $y\geq y_0+2x_1x_2/n$, as well as for all $(x,y)$ with $y\leq y_0-1$ and $x\geq x_0 + 2|y_1y_2|/n$.

We now carefully examine the geometry near the point $(x_0,y_0)$. Proposition \ref{prp:twolines} indicates that at vertical component $y_0-1$, the graph of  $L$ has horizontal component $x_0+2|y_1y_2|/n$, and the horizontal component of the boundary must satisfy 
\[
x_0-4  < x < x_0+ 4\left(1 +\frac{|y_1y_2|}{n}\right) <x_0+8.
\]
Similarly, at horizontal component $x_0-1$, the vertical component is bounded by
\[
y_0 - 4 < y < y_0 + 4\left(1+\frac{x_1x_2}{n}\right) < y_0 + 8.
\]
Thus we see the boundary of $A_\lam$ must pass through the rectangle given by $(x_0-4,x_0+8)\times (y_0-4, y_0+8)$, and so Proposition \ref{prp:twolines} can now be stated in the following form.

\begin{corollary}\label{cor:distance6}
For all $0\leq x \leq y_1-y_2$,
\[
 \dis( L(x), \partial A_{\lambda} ) < 8.
\]
\end{corollary}

For example, with $a=25$, $n=210$ and $N=211$ with $w(n,a/N)$ as shown in Figure \ref{fig:parallel}, we have $\xv = \langle 17,3\frac{210}{211}\rangle$, $\yv = \langle 8, -11\frac{210}{211}\rangle$, and therefore we can compute
\[
 L(x) = \begin{cases}
  \frac{4217}{211}-\frac{102}{211} x & \mbox{for $0\leq x \leq \frac{622}{163}$,} \\
\frac{1999}{88}-\frac{211}{176}x & \mbox{for $\frac{622}{163} \leq x \leq \frac{3998}{211}$}.
\end{cases}
\]
In Figure \ref{fig:exampleboundary} we see $L(x)$ superimposed on $A_{\lambda}$.

\begin{figure}
\[
\begin{tikzpicture}[scale=.4]
\draw (0,22)--(0,0)--(22,0);
\draw (18,0)--(18,3)--(16,3)--(16,5)--(14,5)--(14,7)--(12,7)--(12,10)--(10,10)--(10,12)--(8,12)--(8,14)--(6,14)--(6,16)--(4,16)--(4,19)--(0,19);
\foreach \x in {1,2,3}
{
 \draw (0,\x)--(18,\x);
 \draw (0,7+\x)--(12,7+\x);
 \draw (0,16+\x)--(4,16+\x);
 \draw (\x,0)--(\x,19);
}
\foreach \x in {1,2}
{
 \draw (0,3+\x)--(16,3+\x);
 \draw (0,5+\x)--(14,5+\x);
 \draw (0,10+\x)--(10,10+\x);
 \draw (0,12+\x)--(8,12+\x);
 \draw (0,14+\x)--(6,14+\x);
 \draw (\x+4,0)--(\x+4,16);
 \draw (\x+6,0)--(\x+6,14);
 \draw (\x+8,0)--(\x+8,12);
 \draw (\x+10,0)--(\x+10,10);
 \draw (\x+12,0)--(\x+12,7);
 \draw (\x+14,0)--(\x+14,5);
 \draw (\x+16,0)--(\x+16,3);
}
\draw (4,0)--(4,16);
\draw[line width=2,blue] (0,19.986)--(3.816,18.167)--(18.95,0);
\draw (0,20) node[left] {$x_1+x_2$};
\draw (4,18) node[above right] {$(x_0,y_0)$};
\draw (19,0) node[below] {$y_1-y_2$};
\end{tikzpicture}
\]
\caption{The planar set $A_{\lambda}$ for $w(210, \frac{25}{211})$, along with the piecewise linear curve $L(x)$.}\label{fig:exampleboundary}
\end{figure}
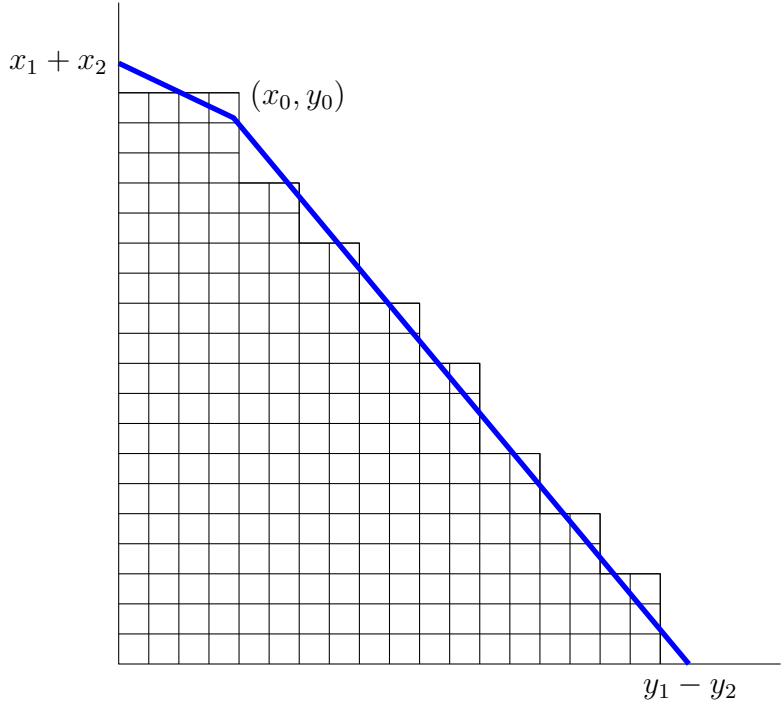

\begin{remark}\label{rem:betashift}
As mentioned in the introduction, the paper \cite{BKPT} actually studied the slightly more general set of permutations generated by fractional parts of lines $f_{\alpha,\beta}(i) = \alpha i + \beta \mod 1$, depending on two independent real parameters. As discussed in \cite{BKPT}, it is easily shown that varying $\beta$ (for fixed $n$ and $\alpha$) acts by cyclic shifting of the \sos{} permutation. Intuitively, a cyclic shift has little effect on monotone subsequences. From our perspective, the vertical shift by $\beta$ amounts to viewing the lattice $L[\xv,\yv]$ within the box $[0,n]\times [\beta', n+\beta']$, where $\beta'=n\beta$. All our definitions, lemmas, and constructions are easily adapted to this shift. We leave details to the interested reader.
\end{remark}

\subsection{Proofs for irrational $\alpha$}\label{sec:proof}

We have already noted that Proposition \ref{prp:l+bounds} implies Theorem \ref{thm:BSarm} for rational numbers. Similarly, by substituting the relevant expressions for $x_1, x_2, y_1, y_2$ from \eqref{xyn_formulas_jodd} or \eqref{xyn_formulas_jeven} into Proposition \ref{prp:twolines} and Corollary \ref{cor:distance6}, we prove Theorems \ref{thm:lam_bounds} and \ref{thm:main}, respectively, for rational numbers. We now show how to extend those results to all reals. 

Fix $n\in \N$ and an irrational number $0<\alpha<1$. Denote the sequence of convergents for $\alpha$ as  $p_0/q_0, p_1/q_1, \dots$, and let $i$ be the smallest non-negative integer such that 
\eq\label{eq:alpha_conv}
\delta_i = \left|\alpha - \frac{p_i}{q_i} \right| \le \frac{1}{n}.
\eeq
If $i=2h+2$ is even, then  $\frac{p_i}{q_i} < \alpha$, and we have
\[
\frac{p_{2h+1}}{q_{2h+1}}-\alpha \ge \frac{1}{n}, \qquad 0<\alpha-\frac{p_{2h+2}}{q_{2h+2}} < \frac{1}{n}.
\]
In this case we apply the program of Section \ref{sec:fulllattice} with $a= p_{m}$ and $N=q_{m}$, where $m$ is an {\it even} integer satisfying $q_{m}>n$.
Note then that $\frac{p_{m}}{q_{m}} <\alpha$, so we have
\[
\frac{p_{2h+1}}{q_{2h+1}}-\frac{p_{m}}{q_{m}}> \frac{p_{2h+1}}{q_{2h+1}}-\alpha \ge \frac{1}{n}, \qquad 0<\frac{p_{m}}{q_{m}}-\frac{p_{2h+2}}{q_{2h+2}}<\alpha-\frac{p_{2h+2}}{q_{2h+2}} < \frac{1}{n},
\]
which implies that the convergents appearing in \eqref{xyn_formulas_jodd} are the same as those defined by \eqref{eq:alpha_conv}. 

Similarly, if $i=2h+1$ is odd, then we take $a = p_{m}$ and $N=q_{m}$, where $m$ is an {\it odd} integer satisfying $q_{m}>n$ to yield that the convergents appearing in \eqref{xyn_formulas_jeven} are the same as those defined by \eqref{eq:alpha_conv}. 

With these choices of $a$ and $N$, all the results of Section \ref{sec:fulllattice} can be applied to $\alpha$. Let $\alpha(m)=p_m/q_m$ denote the convergent we have chosen. Since $\alpha(m)$ is a has denominator $q_m>n$, both $\alpha$ and $\alpha(m)$ belong to the same Farey interval in $F^{(n)}$. But then Theorem \ref{thm:suranyi} tells us the corresponding \sos\, permutations are equal: $w(n,\alpha)=w(n,\alpha(m))$. Furthermore, these numbers $\alpha$ and $\alpha(m)$ induce same Schensted shape $A_{\lambda}$. 

For the rational approximation $\alpha(m)$, let 
\[
\delta_i(m)= \left|\alpha(m) - \frac{p_i}{q_i}\right|.
\]
For fixed $n$, $m$ can be chosen arbitrarily large, provided $m$ has the desired parity and $q_m>n$. Since our results hold for any such $m$, we may then take the limit as $m\to\infty$, in which case $\delta_i(m)\to \delta_{i}$. This proves Theorems \ref{thm:BSarm}, \ref{thm:lam_bounds}, and \ref{thm:main} for any real $\alpha$.

\section{Further directions}\label{sec:further}

We hope the results in this paper can serve as a starting point for further study. Here are a few directions they could lead.

\begin{itemize}
\item As described in the {\bf Big Question} in the introduction, our original motivation was to study the Schensted shape of {\it random} \sos\ permutations. We currently cannot describe the probability distribution on partitions which arises from the uniform distributions on $\sos_n$ via Schensted insertion. It would be interesting to see if there is an average shape which has a scaling limit as $n\to\infty$.

\item The set $\sos_n$ represents a rather extreme restriction of the full symmetric group, which is the reason for the markedly different quality of the Schensted shape of a \sos\ permutation compared to a random one. Numerical experiments indicate that other sequences of the form $(f(1), f(2), \dots, f(n))$ have Schensted shapes similar to that of a uniformly random permutation (in all of $S_n$) for a fairly general class of nonlinear functions $f$, including nonlinear polynomials. If $f$ is chosen to interpolate between a nonlinear function and a linear one as $n\to\infty$, it is very plausible that the Schensted shape would interpolate between the Logan--Shepp/Vershik--Kerov limit shape and the piecewise linear one studied in this paper. For example, one could consider  $f(i) =n^{-\gamma}\beta i^2+ \alpha i \mod 1$ for fixed irrational numbers $\alpha$ and $\beta$ and an appropriate scaling exponent $\gamma>0$.
\item For pattern avoiding permutations, there are some results on the distribution of the longest increasing subsequence as $n\to\infty$ \cite{DHW, MinerPak}. One could also consider their full Schensted shape. The paper \cite{MinerPak} considers a different notion of shape, namely the density of points on the graph $(i, w(i))$.
\end{itemize}

\bibliographystyle{plain}
\bibliography{bibliography}

\def\cydot{\leavevmode\raise.4ex\hbox{.}}
\begin{thebibliography}{10}

\bibitem{AlessandriBerthe}
Pascal Alessandri and Val\'{e}rie Berth\'{e}.
\newblock Three distance theorems and combinatorics on words.
\newblock {\em Enseign. Math. (2)}, 44(1-2):103--132, 1998.

\bibitem{BiringerSchmidt}
Ian Biringer and Benjamin Schmidt.
\newblock The three gap theorem and {R}iemannian geometry.
\newblock {\em Geom. Dedicata}, 136:175--190, 2008.

\bibitem{Bleher1}
P.~M. Bleher.
\newblock The energy level spacing for two harmonic oscillators with golden
  mean ratio of frequencies.
\newblock {\em J. Statist. Phys.}, 61(3-4):869--876, 1990.

\bibitem{Bleher2}
P.~M. Bleher.
\newblock The energy level spacing for two harmonic oscillators with generic
  ratio of frequencies.
\newblock {\em J. Statist. Phys.}, 63(1-2):261--283, 1991.

\bibitem{BKPT}
Sarah Bockting-Conrad, Yevgenia Kashina, T.~Kyle Petersen, and Bridget~Eileen
  Tenner.
\newblock S\'{o}s {P}ermutations.
\newblock {\em Amer. Math. Monthly}, 128(5):407--422, 2021.

\bibitem{BoydSteele}
David~W. Boyd and J.~Michael Steele.
\newblock Monotone subsequences in the sequence of fractional parts of
  multiples of an irrational.
\newblock {\em J. Reine Angew. Math.}, 306:49--59, 1979.

\bibitem{dJS}
A.~del Junco and J.~Michael Steele.
\newblock Growth rates for monotone subsequences.
\newblock {\em Proc. Amer. Math. Soc.}, 71(2):179--182, 1978.

\bibitem{DHW}
Emeric Deutsch, A.~J. Hildebrand, and Herbert~S. Wilf.
\newblock Longest increasing subsequences in pattern-restricted permutations.
\newblock {\em Electron. J. Combin.}, 9(2):Research paper 12, 8, 2002/03.
\newblock Permutation patterns (Otago, 2003).

\bibitem{ErdosSzekeres35}
P.~Erd\"{o}s and G.~Szekeres.
\newblock A combinatorial problem in geometry.
\newblock {\em Compositio Math.}, 2:463--470, 1935.

\bibitem{Greene}
Curtis Greene.
\newblock An extension of {S}chensted's theorem.
\newblock {\em Advances in Math.}, 14:254--265, 1974.

\bibitem{Hammersley72}
J.~M. Hammersley.
\newblock A few seedlings of research.
\newblock In {\em Proceedings of the {S}ixth {B}erkeley {S}ymposium on
  {M}athematical {S}tatistics and {P}robability ({U}niv. {C}alifornia,
  {B}erkeley, {C}alif., 1970/1971), {V}ol. {I}: {T}heory of statistics}, pages
  345--394, 1972.

\bibitem{HardyWright}
G.~H. Hardy and E.~M. Wright.
\newblock {\em An introduction to the theory of numbers}.
\newblock Oxford University Press, Oxford, sixth edition, 2008.
\newblock Revised by D. R. Heath-Brown and J. H. Silverman, With a foreword by
  Andrew Wiles.

\bibitem{Hartman}
S.~Hartman.
\newblock \"{U}ber die {A}bst\"{a}nde von {P}unkten {$n\xi$} auf der
  {K}reisperipherie.
\newblock {\em Ann. Soc. Polon. Math.}, 25:110--114 (1953), 1952.

\bibitem{LoganShepp}
B.~F. Logan and L.~A. Shepp.
\newblock A variational problem for random {Y}oung tableaux.
\newblock {\em Advances in Math.}, 26(2):206--222, 1977.

\bibitem{MinerPak}
Sam Miner and Igor Pak.
\newblock The shape of random pattern-avoiding permutations.
\newblock {\em Adv. in Appl. Math.}, 55:86--130, 2014.

\bibitem{Narushima}
Terumi Narushima.
\newblock {\em Microtonality and the Tuning Systems of Erv Wilson: Mapping the
  Harmonic Spectrum}.
\newblock Routledge, 2017.

\bibitem{Romik}
Dan Romik.
\newblock {\em The surprising mathematics of longest increasing subsequences},
  volume~4 of {\em Institute of Mathematical Statistics Textbooks}.
\newblock Cambridge University Press, New York, 2015.

\bibitem{Schensted}
C.~Schensted.
\newblock Longest increasing and decreasing subsequences.
\newblock {\em Canadian J. Math.}, 13:179--191, 1961.

\bibitem{Sos}
V.~S\'os.
\newblock On the distribution mod 1 of the sequence $n\alpha$.
\newblock {\em Ann. Univ. Sci. Budapest E\"otv\"os Sect. Math.}, 1:127--134,
  1958.

\bibitem{Steele95}
J.~Michael Steele.
\newblock Variations on the monotone subsequence theme of {E}rd\H{o}s and
  {S}zekeres.
\newblock In {\em Discrete probability and algorithms ({M}inneapolis, {MN},
  1993)}, volume~72 of {\em IMA Vol. Math. Appl.}, pages 111--131. Springer,
  New York, 1995.

\bibitem{Suranyi}
J.~Sur\'anyi.
\newblock {\"U}ber die {A}nordnung der {V}ielfachen einer reellen {Z}ahl mod 1.
\newblock {\em Ann. Univ. Sci. Budapest E\"otv\"os Sect. Math.}, 1:107--111,
  1958.

\bibitem{Swierczkowski}
S.~\'{S}wierczkowski.
\newblock On successive settings of an arc on the circumference of a circle.
\newblock {\em Fund. Math.}, 46:187--189, 1959.

\bibitem{vanRavenstein}
Tony van Ravenstein.
\newblock The three gap theorem ({S}teinhaus conjecture).
\newblock {\em J. Austral. Math. Soc. Ser. A}, 45(3):360--370, 1988.

\bibitem{Ravenstein}
Tony~Peter van Ravenstein.
\newblock {\em Number sequences and phyllotaxis}.
\newblock ProQuest LLC, Ann Arbor, MI, 1986.
\newblock Thesis (Ph.D.)--University of Wollongong (Australia).

\bibitem{VershikKerov}
A.~M. Ver\v{s}ik and S.~V. Kerov.
\newblock Asymptotic behavior of the {P}lancherel measure of the symmetric
  group and the limit form of {Y}oung tableaux.
\newblock {\em Dokl. Akad. Nauk SSSR}, 233(6):1024--1027, 1977.

\end{thebibliography}

\end{document}